\documentclass[12pt]{amsart} 
\usepackage{amsmath, amssymb, mathscinet,mathtools}
\usepackage{amsthm} 
\usepackage[all]{xy}
\usepackage{enumerate}
\usepackage[T1]{fontenc}
\usepackage{hyperref} 
\usepackage{mathrsfs} 
\usepackage{graphicx}
\usepackage{color}
\renewcommand{\thefootnote}{} 
\theoremstyle{plain} 
\newtheorem{theorem}{\indent\sc Theorem}[section]
\newtheorem{lemma}[theorem]{\indent\sc Lemma}
\newtheorem{corollary}[theorem]{\indent\sc Corollary}
\newtheorem{proposition}[theorem]{\indent\sc Proposition}

\theoremstyle{definition} 
\newtheorem{definition}[theorem]{\indent\sc Definition}
\newtheorem{remark}[theorem]{\indent\sc Remark}

\newcommand{\R}{\mathbb{R}}

\newcommand{\N}{\mathbb{N}}

\def\dim{\mathop{\mathrm{dim}}\nolimits}

\def\Im{\mathop{\mathrm{Im}}\nolimits}

\newcommand{\abs}[1]{\left\lvert#1\right\rvert}

\newcommand{\diam}{\mathop{\mathrm{diam}}\nolimits}

\newcommand{\id}{\mathop{\mathrm{id}}}

\newcommand{\calL}{\mathcal{L}}
\newcommand{\xcalL}[1]{{{}_{#1}\mathcal{L}}}
\newcommand{\calO}{\mathcal{O}}

\newcommand{\OdaX}{\calO \partial_a X}
\newcommand{\OdbY}{\calO \partial_b Y}
\newcommand{\loge}{\log^\epsilon}
\newcommand{\expe}{\exp_\epsilon}

\newcommand{\radcon}{\phi_{r}}
\newcommand{\radconi}{\phi}
\newcommand{\ds}[1]{\,\overline{#1}\,}
\newcommand{\Rp}{\R_{\geq 0}}

\newcommand{\DA}{\Omega}
\newcommand{\Dom}{\operatorname{Dom}}
\newcommand{\hyp}{\mathbb{H}}
\newcommand{\cay}{\mathrm{Cay}}
\newcommand{\dbdl}[2]{d_{#1,#2}}
\newcommand{\rad}{\mathop{\mathrm{rad}}\nolimits}
\newcommand{\radp}{\mathop{\mathrm{rad'}}\nolimits}
\newcommand{\Xvis}{X^{\text{Vis}}}
\newcommand{\Zvis}{Z^{\text{Vis}}}
\makeatletter
\def\address#1#2{\begingroup
\noindent\parbox[t]{7.8cm}{%
\small{\scshape\ignorespaces#1}\par\vskip1ex
\noindent\small{\itshape E-mail address}%
\/: #2\par\vskip4ex}\hfill%
\endgroup}%
\makeatother
\title[]{\uppercase{Visual maps between coarsely convex spaces}}
\author{
\textsc{Yuuhei Ezawa, Tomohiro Fukaya} 
}
\date{} 
\begin{document}



\thanks{T.Fukaya was supported by JSPS KAKENHI Grant Number JP19K03471}
\renewcommand{\thefootnote}{\fnsymbol{footnote}} 

\begin{abstract}
The class of coarsely convex spaces is a coarse geometric analogue of
the class of nonpositively curved  Riemannian manifolds. It includes
Gromov hyperbolic spaces, CAT(0) spaces, proper injective metric spaces
and systolic complexes.
It is well known that quasi-isometric embeddings of Gromov hyperbolic spaces induce topological embeddings of their boundaries. 
Dydak and Virk studied maps between Gromov hyperbolic spaces which
induce continuous maps between their boundaries. In this paper, we
generalize their work to maps between coarsely convex spaces.
\end{abstract}

\maketitle

\section{Introduction}
\label{sec:introduction} The class of coarsely convex spaces is
introduced by the second author and Shin-ichi Oguni~\cite{FO-CCH}.
This class includes Gromov hyperbolic spaces, and
Busemann nonpositively curved spaces, especially, CAT(0) spaces.
Moreover, it is shown that systolic complexes and proper injective
metric spaces are coarsely convex spaces, by Osajda and
Przytycki~\cite{bdrySytolic}, respectively, by Descombes and
Lang~\cite{convex-bicomb}.
Recently, 
Chalopin, Chepoi, Genevois, Hirai and Osajda \cite{chalopin2020helly-arXiv}
showed that Helly graphs and coarsely Helly spaces are 
coarsely dense in their injective hulls,
especially, they are coarsely convex.
This result may become a rich source of coarsely convex
spaces. In fact, Osajda and Huang~\cite{LArtin-systolic} 
showed that Artin groups of type FC and
Garside groups act geometrically on Helly graphs, hence they do
on coarsely convex spaces. 
Haettel, Hoda and Petyt~\cite{haettel2020coarse}
showed that any
hierarchically hyperbolic space admits a new metric that is coarsely
Helly, so it is coarsely convex.

Let $X$ be a proper coarsely convex space. In~\cite{FO-CCH}, the authors
constructed the ideal boundary $\partial X$ of $X$, and they showed that
the open cone $\calO \partial X$ is coarsely homotopy equivalent to $X$
via the ``exponential map'' and the ``logarithmic map''. As a corollary,
it follows that the coarse Baum-Connes conjecture holds for coarsely
convex spaces. In this paper, we consider maps between coarsely convex spaces which induce continuous maps between their boundaries.

Dydak and Virk \cite{dydak-G-bdry-maps}
introduced the class of visual maps between Gromov hyperbolic spaces
that induce continuous maps between their boundaries. They also introduced its subclass, called radial maps, and showed that these maps
induce H\"older maps between boundaries.

Then they showed that radial coarsely $n$-to-one maps between proper
geodesic Gromov hyperbolic spaces induces continuous $n$-to-one maps
between boundaries. As a
consequence, they obtained dimension rising theorem for radial
maps, with respect to the topological dimensions of the Gromov
boundaries, and to the asymptotic dimensions of hyperbolic groups.

A metric space $(X,d)$ is coarsely convex if 
$X$ has a family of quasi-geodesics $\xcalL{X}$, 
called the \textit{system of good quasi-geodesics}, 
satisfying some conditions.
Especially, along quasi-geodesics in $\xcalL{X}$, the metric $d$ 
satisfies \textit{coarsely convex} inequality 
(Definition~\ref{def:cBNC}(\ref{qconvex})${^q}$).

In this paper, we introduce the class of visual maps between coarsely convex spaces, that do induce continuous maps between 
their ideal boundaries (Proposition~\ref{prop:visual-extension} and Corollary~\ref{cor:visu-bdl-cont}). 

We present a fixed point property of the maps induced by the visual isometries. Let $M$ be a metric space and 
$f\colon M\to M$ be an isometry. We say that $f$ is \textit{elliptic}
if the set $\{f^n(x)\colon n\in \N\}$ is bounded for some $x\in M$.

\begin{theorem}
\label{thm:fixed-pt}
 Let $X$ be a proper coarsely convex space and
 let $f\colon X\to X$ be an isometry. We suppose that $f$ is visual. Then,
 either $f$ is elliptic, or, the induced map 
 $\partial f\colon \partial X\to \partial X$ has a fixed point.
\end{theorem}

Let $X$ and $Y$ be coarsely convex spaces with systems of 
good quasi-geodesics $\xcalL{X}$ and $\xcalL{Y}$, respectively.
We also introduce the class of
$\xcalL{X}$-radial and $\xcalL{X}$-$\xcalL{Y}$-equivariant maps, 
which is in fact a subclass of visual maps (Theorem~\ref{thm:visual-Lip}).

The following is a generalization of the large scale dimension raising
theorem for hyperbolic spaces by Dydak and Virk~\cite[Theorem 1.1]{dydak-G-bdry-maps}.
\begin{theorem}
\label{thm:dim-rise}
 Let $X$ and $Y$ be proper coarsely convex spaces, and let $\xcalL{X}$
 and $\xcalL{Y}$ be systems of good quasi-geodesics of $X$ and $Y$,
 respectively.  Let $f\colon X\to Y$ be a large scale Lipschitz map which is
 $\xcalL{X}$-radial and 
 $\xcalL{X}$-$\xcalL{Y}$-equivariant. If $f$ is coarsely $(n+1)$-to-one 
 and coarsely surjective then
 \begin{align*}
  \dim (\partial Y)\leq \dim (\partial X) + n.
 \end{align*}
\end{theorem}

We also study extension problems. 
Let $f\colon \partial X\to \partial Y$ be a continuous map between
ideal boundaries of coarsely convex spaces. Then we construct 
a visual map $\rad f\colon X\to Y$, called radial extension of $f$,
which induces the given map $f$ of the boundaries.
We also show that for visual
$\xcalL{X}^\infty$-$\xcalL{Y}^\infty$-equivariant map
$F\colon X\to Y$, the radial extension of its induced map 
$\partial F\colon \partial X\to \partial Y$ is coarsely homotopic to
the given map $F$. These results are summarized in 
Theorem~\ref{thm:rad-ext}.

The organization of the paper is as follows. In
Section~\ref{sec:coars-conv-space}, we review the definition and basic
properties of coarsely convex spaces.  In
Section~\ref{sec:visual-map-radial}, we introduce the class of visual
maps and its subclasses. At the end of the section, we give a 
proof of Theorem~\ref{thm:fixed-pt}.

In Section~\ref{sec:coarsely-surjective}, we
study on coarsely surjective maps. In Section~\ref{sec:coarsely-n-one},
we show that coarsely $n$-to-one maps induce $n$-to-one maps 
between boundaries (Theorem~\ref{thm:coarseN-to-one}), 
and give a proof of Theorem~\ref{thm:dim-rise}. 
In Section~\ref{sec:example}, we discuss maps
between direct products of coarsely convex spaces and give an explicit
example of a map satisfying the assumptions in Theorem~\ref{thm:dim-rise}
whose domain and range are not Gromov hyperbolic spaces. 

In Section~\ref{sec:radial-extension}, 
we give the construction of the radial extension, and give a proof of 
Theorem~\ref{thm:rad-ext}.




\section{Coarsely convex space}
\label{sec:coars-conv-space}
In this section, we review the theory of coarsely convex spaces and their boundaries.

Throughout this paper, we use the following notations.  For points $x,y$
in a metric space $X$, we denote by $\ds{x,y}$ the distance between $x$
and $y$. For a subset $A\subset X$ and a point $x\in X$, we denote by
$\ds{x,A}$ the distance between $x$ and $A$. We denote by $\diam A$ the
diameter of $A$.
For a map $F\colon P\to Q$,
we denote by $\Dom F$ the domain of $F$, that is, $\Dom F = P$. 
For $a,b\in \R$, we denote by $a\vee b$ the maximum of $\{a,b\}$.

\subsection{Coarse map}
Let $X,Y$ be metric spaces. Let $f\colon X\to Y$ be a map.
\begin{enumerate}
 \item 
       The map $f$ is \textit{bornologous}
       if there exists a non-decreasing function
       $\psi\colon \R_{\geq 0} \to \R_{\geq 0}$ such that 
       for all $x,x'\in X$, we have
       \[
       \ds{f(x),f(x')} \leq \psi(\ds{x,x'}).
       \]
       Especially, for constants $A\geq 1,\, B\geq 0$, 
       $f$ is a \textit{$(A,B)$-large scale Lipschitz} map if 
       for all $x,x'\in X$, we have $\ds{f(x),f(x')}\leq A\ds{x,x'}+B$.
 \item The map $f$ is \textit{metrically proper} if for each bounded subset
       $B\subset Y$, the inverse image $f^{-1}(B)$ is bounded.
 \item The map $f$ is \textit{coarse} if it is bornologous and 
       metrically proper.
 \item The map $f$ is a \textit{coarse embedding} if it is bornologous and there exists 
       a non-decreasing function $\psi'\colon \Rp\to\Rp$ such that 
       $\psi'(t)\to \infty\, (t\to \infty)$ and for all $x,x'\in X$, we have
       \[
       \psi'(\ds{x,x'})\leq \ds{f(x),f(x')}.
       \]
 \item The map $f$ is \textit{coarsely surjective} if there exists $C\geq 0$ 
       such that $C$-neighbourhood of $f(X)$ is $Y$.
 \item The map $f$ is a \textit{coarse equivalence map} if it is 
       a coarse embedding and coarsely surjective.
\end{enumerate}
For more details on the above notions, we refer~\cite{MR2007488}.


\subsection{Coarse homotopy}

 Let $f,g\colon X\to Y$ be coarse maps between metric
 spaces. The maps $f$ and $g$ are 
 {\itshape coarsely homotopic} if there exists a metric subspace 
$Z = \{(x,t):0\leq t\leq T_x\}$ 
of $X\times \R_{\geq 0}$ and a coarse map
 $h\colon Z\to Y$, such that
\begin{enumerate}
 \item the map $X \ni x\mapsto T_x\in \R_{\geq 0}$ is bornologous,
 \item $h(x,0) = f(x)$, and
 \item $h(x,T_x) = g(x)$.
\end{enumerate}
Here we equip $X\times \R_{\geq 0}$ with the $l_1$-metric, that is, 
$d_{X\times \R_{\geq 0}}((x,t),(y,s)):= \ds{x,y}+ \abs{t-s}$ for 
$(x,t),(y,s)\in X\times \R_{\geq 0}$.


\subsection{Coarsely convex space}
 
\begin{definition}
\label{def:cBNC}
Let $X$ be a metric space. Let $\lambda\geq 1$, $k \geq 0$, $E\geq 1$,
and $C\geq 0$ be constants. 
Let $\theta\colon \Rp\to\Rp$ be a non-decreasing function.
Let $\calL$ be a family of $(\lambda,k)$-quasi-geodesic segments.  
The metric space $X$ is {\itshape
$(\lambda,k,E,C,\theta,\calL)$-coarsely convex}, if
$\calL$ satisfies the following.
\begin{enumerate}[(i)$^q$] 
 \item \label{qconn}
       For $v,w\in X$, there exists a 
       quasi-geodesic segment
       $\gamma\in \calL$ with $\Dom \gamma = [0,a]$,
       $\gamma(0) = v$ and $\gamma(a) = w$.
 \item \label{qconvex}
       Let $\gamma, \eta \in \calL$ be 
       quasi-geodesic segments 
       with $\Dom \gamma = [0,a]$ and $\Dom \eta = [0,b]$. 
       Then for 
       $t\in [0,a]$, $s\in [0,b]$, and $0\leq c\leq 1$, we have that 
       \begin{align*}
	\ds{\gamma(ct),\eta(cs)}
	\leq cE\ds{\gamma(t),\eta(s)} 
	      + (1-c)E\ds{\gamma(0), \eta(0)}+ C.
       \end{align*}
 \item \label{qparam-reg}       
       Let $\gamma,\eta\in \calL$ be quasi-geodesic segments 
       with $\Dom \gamma = [0,a]$ and $\Dom \eta = [0,b]$. 
       Then for $t\in [0,a]$ and $s\in [0,b]$, we have
       \begin{align*}
	\abs{t-s} \leq \theta(\ds{\gamma(0),\eta(0)}+\ds{\gamma(t),\eta(s)}).
       \end{align*}
\end{enumerate} 
The family $\calL$ satisfying (\ref{qconn})$^q$, 
(\ref{qconvex})$^q$, and  (\ref{qparam-reg})$^q$
is called a {\itshape system of good quasi-geodesic segments}, and elements 
 $\gamma\in \calL$ are called {\itshape good quasi-geodesic segments}.
We call the inequality in (\ref{qconvex})$^q$ the {\itshape coarsely convex inequality}.
\end{definition}

We say that a metric space $X$ is a {\itshape coarsely convex} space if
there exist constants $\lambda,k,E,C$,
a non-decreasing function $\theta\colon \Rp\to \Rp$, 
and a family of
$(\lambda,k)$-quasi-geodesic segments $\calL$ such that $X$ is
$(\lambda,k,E,C,\theta,\calL)$-coarsely convex.  

We remark that if $\calL$ consists of only geodesic segments,
then $\calL$ satisfies~(\ref{qparam-reg})$^q$ by the triangle
inequality. We say that a metric space is {\itshape geodesic $(C,\calL)$-coarsely convex}
if it is $(1,0,1,C,\mathrm{id}_{\Rp},\calL)$-coarsely convex.

The class of coarsely convex spaces is closed under direct products.
Let $X$ and $Y$ be metric spaces and $\gamma\colon [0,a]\to X$ 
and $\eta\colon [0,b]\to Y$ be maps. Then we define the map 
$\gamma\oplus \eta\colon [0,a+b]\to X\times Y$ by setting
\begin{align*}
 \gamma\oplus \eta(t):=   \left(\gamma\left(\frac{a}{a+b}t\right),
  \eta\left(\frac{b}{a+b}t\right)\right), \quad t\in [0,a+b].
\end{align*}
If $\gamma$ and $\eta$ are quasi-geodesic segments, so is $\gamma\oplus \eta$.

\begin{proposition}[{\cite{FO-CCH}}]
 \label{prop:productspace} Let $(X,d_X)$ and $(Y,d_Y)$ be coarsely
 convex metric spaces with systems of good quasi-geodesic segments $\xcalL{X}$ and
 $\xcalL{Y}$, respectively.
 Then the product with the $\ell_1$-metric
 $(X\times Y, d_{X\times Y})$ is coarsely convex, whose 
 system of good quasi-geodesic segments $\xcalL{X\times Y}$ consists of
 all $\gamma\oplus \eta$ for $\gamma\in \xcalL{X}$ and $\eta\in \xcalL{Y}$.
 \end{proposition}

In the rest of this section, let $X$ be a $(\lambda,k,E,C,\theta,\calL)$-coarsely convex space.

\subsection{Ideal boundary}
Let $\gamma\colon \Rp \to X$ be a
map.  Let $\gamma_n\colon [0,a_n]\to X$ be
quasi-geodesic segments in $X$.  A sequence $\{(\gamma_n,a_n)\}_{n}$ 
with $a_n\to \infty$
is
an {\itshape $\calL$-approximate} sequence for $\gamma$
if for all $n$, we have
$\gamma_n\in \calL$, $\gamma_n(0)=\gamma(0)$ and
$\{\gamma_n\}_{n}$ converges to $\gamma$ pointwise on $\N$.
A map $\gamma\colon\Rp\to X$ 
is {\itshape $\calL$-approximatable} if
there exists an $\calL$-approximate sequence for $\gamma$.

We define a family of quasi-geodesic rays, denoted by $\calL^\infty$, 
as a family consisting of all $\calL$-approximatable maps 
$\gamma\colon \Rp\to X$ such that 
$\gamma(t)=\gamma(\lfloor t\rfloor)$ for all $t\in \Rp$.
We set $\bar{\calL}:= \calL\cup \calL^\infty$. 
We remark that $\gamma\in\calL^\infty$ is a $(\lambda,k_1)$-quasi-geodesic ray, 
where $k_1:= \lambda+k$. See \cite[Lemma 4.1]{FO-CCH}.

\begin{proposition}[{\cite[Proposition 4.2]{FO-CCH}}]
\label{prop:qgeod-ray-convex}
The family $\bar{\calL}$ satisfies the following.
 \begin{enumerate}[\hspace{7mm}$(1)$]
 \item \label{qrayconvex}       
       Let $\gamma, \eta \in \bar{\calL}$ be quasi-geodesics.
       Then for $t\in \Dom \gamma$, $s\in \Dom \eta$ and $0\leq c\leq 1$, we have 
       \begin{align*}
	\ds{\gamma(ct),\eta(cs)}
	\leq cE\ds{\gamma(t),\eta(s)} +(1-c)E\ds{\gamma(0),\eta(0)}+ D,
       \end{align*}
       	where $D:= 2(1+E)k_1+ C$.
 \item \label{qrayparam-reg} We define a non-decreasing function
       $\tilde{\theta}\colon \Rp\to \Rp$ by 
       $\tilde{\theta}(t):=\theta(t+1)+1$.
       For $\gamma,\eta\in \bar{\calL}$ and for $t\in \Dom \gamma$, $s\in \Dom \eta$, we have
       \begin{align*}
	\abs{t-s}\leq 
	\tilde{\theta}(\ds{\gamma(0),\eta(0)}+\ds{\gamma(t),\eta(s)}).
       \end{align*}       
 \end{enumerate}
\end{proposition}

\begin{lemma}[{\cite[Lemma 4.3]{FO-CCH}}]
\label{lem:t-ab}
 Let $\gamma,\eta\in \bar{\calL}$ be quasi-geodesics such that
 $\gamma(0)=\eta(0)$.
 For all $a\in \Dom \gamma$, $b\in \Dom \eta$ and 
 $0\leq t\leq \min\{a,b\}$, we have
 \begin{align*}
  \ds{\gamma(t),\eta(t)}
  &\leq E(\ds{\gamma(a),\eta(b)} 
    + \lambda\tilde{\theta}(\ds{\gamma(a),\eta(b)}) + k_1) + D.
 \end{align*} 
\end{lemma}

 For quasi-geodesic rays $\gamma$ and $\eta$ in $\calL^\infty$, 
 we say that $\gamma$ and $\eta$ are {\itshape equivalent} if
 \begin{align*}
 \sup\{\ds{\gamma(t),\eta(t)}:t\in \Rp\}<\infty,
 \end{align*}
 and we denote by $\gamma \sim \eta$. 
 For $\gamma\in \calL^\infty$, we denote by $[\gamma]$ its equivalence
 class.
 The {\itshape ideal boundary} of $X$ is the
 set $\partial X:= \calL^\infty/ \sim$ of equivalence classes of
 quasi-geodesic rays in $\calL^\infty$.

 Let $O\in X$ be a base point. We define $\calL_O^\infty$ as
 the subset of $\calL^\infty$ consisting of 
 all quasi-geodesic rays in $\calL^\infty$ stating at $O$.
 The {\itshape ideal boundary} of $X$ with respect to $O$ is the
 set $\partial_O X:= \calL_O^\infty/ \sim$ of equivalence classes of
 quasi-geodesic rays in $\calL_O^\infty$.

\begin{lemma}[{\cite[Lemma 4.5]{FO-CCH}}]
\label{lem:HdistD}
 For $\gamma, \eta\in \calL_O^\infty$, 
 $\gamma$ and $\eta$ are equivalent if and only if $\ds{\gamma(t),\eta(t)}\leq D$ for all
 $t\in \Rp$. 
\end{lemma}
 
We define a subset $\calL_O\subset \calL$ as the set of all $\gamma\in \calL$ with $\Dom \gamma=[0,a_\gamma]$, 
$a_\gamma \geq 2\theta(0)$, and $\gamma(0)= O$. Set $\bar{\calL}_O:= \calL_O \cup \calL_O^\infty$. 

For $x\in X\setminus \bar{B}(O,2\lambda\theta(0)+k)$, set
\begin{align*}
 \bar{\calL}_O(x):=\{\gamma\in \calL_O: \Dom\gamma = [0,a_\gamma], \, \gamma(a_\gamma)=x\},
\end{align*}
and for $x\in \partial_O X$, set
\begin{align*}
 \bar{\calL}_O(x):=\{\gamma\in \calL_O^\infty: \gamma \in x\}.
\end{align*}


\begin{definition}
\label{def:gromov-prod}
 We define a product $(\cdot\mid \cdot)_O\colon \bar{\calL}_O\times \bar{\calL}_O\to \Rp\cup \{\infty\}$ as 
 follows. For $\gamma,\eta\in \bar{\calL}_O$, 
 \begin{align*}
  (\gamma\mid \eta)_O:= \sup\{t\in \Dom\gamma \cap \Dom \eta:\ds{\gamma(t),\eta(t)}\leq 2D+2\}.
 \end{align*}
 We also define a product $(\cdot\mid \cdot)_O\colon (X\cup \partial_O X) \times (X\cup \partial_O X)\to \Rp\cup \{\infty\}$
 as follows.
 \begin{enumerate}[(1)]
  \item If $x\in \bar{B}(O,2\lambda\theta(0)+k))$ or $y\in \bar{B}(O,2\lambda\theta(0)+k))$,
	set
	\begin{align*}
	 (x\mid y)_O :=0.
	\end{align*}
  \item If $x,y\in (X\setminus \bar{B}(O,2\lambda\theta(0)+k))\cup \partial_O X$, set
	\begin{align*}
	 (x\mid y)_O:=\sup\{(\gamma\mid \eta)_O:\gamma\in \bar{\calL}_O(x),\, \eta\in \bar{\calL}_O(y)\}.
	\end{align*}
 \end{enumerate}
 When the choice of the base point is clear, we write $(x\mid y)$ instead of $(x\mid y)_O$.
\end{definition}

We quote some lemmas from \cite{FO-CCH} which we need later.

\begin{lemma}[{\cite[Lemma 4.14]{FO-CCH}}]
\label{lem:prod-approx}
 \label{lem:approx-seg-Gprod} 
 Let $\gamma\in \calL_O^\infty$ be a quasi-geodesic ray and let 
 $\{(\gamma_n,a_n)\}_{n}$ be an $\calL$-approximate sequence 
 for $\gamma$.  Then we have 
 $\liminf_{n\to \infty}(\gamma\mid \gamma_n)= \infty$.	
\end{lemma}

\begin{lemma}
\label{lem:univ-const}
 There exists a constant $\DA\geq 1$ depending on 
 $\lambda,k,E,C,\theta(0)$ such that
 the following holds:
 \begin{enumerate}[$($i$)$]
  \item \label{lem:D3}
 For $x,y\in X\cup \partial_O X$ and for $\gamma\in \bar{\calL}_O(x)$ and $\eta\in \bar{\calL}_O(y)$, we have
   \begin{align*}
   (\gamma\mid \eta) \leq  (x\mid y) \leq \DA(\gamma\mid \eta).
  \end{align*}
  \item \label{lem:qultm}
 For triplet $\gamma,\eta,\xi\in \bar{\calL}_O$, we have
 \begin{align*}
  (\gamma\mid \xi) \geq \DA^{-1}\min
  \{(\gamma\mid \eta), (\eta\mid \xi)\}.
 \end{align*}
  \item \label{lem:qrayultmvisu} 	
	For triplet $x,y,z\in X\cup \partial_O X$, 
 we have 
 \begin{align*}
  (x\mid z) \geq \DA^{-1}\min \{(x\mid y), (y\mid z)\}.
 \end{align*}
  \item \label{lem:maximizer}
 Let $\gamma,\eta\in \bar{\calL}_O$. For all $t\in \Rp$ with 
	$t\leq (\gamma\mid \eta)$, we have
 \begin{align*}
  \ds{\gamma(t), \eta(t)} \leq \DA.
 \end{align*}
  \item \label{lem:ray-same-param}
 Let $\gamma,\eta\in \bar{\calL}_O$.
 If $\gamma(a) = \eta(b)$ for some $a\in \Dom \gamma$ and $b\in \Dom \eta$,
	then for all $t\in \Dom \gamma\cap \Dom \eta$ with
	$0\leq t \leq \max\{a,b\}$, 
	we have
 \begin{align*}
  \ds{\gamma(t),\eta(t)}\leq \DA.
 \end{align*} 
 \end{enumerate}
\end{lemma}

\begin{proof}
(\ref{lem:D3}),(\ref{lem:qultm}) and (\ref{lem:qrayultmvisu}) 
 follow from \cite[Lemma 4.12]{FO-CCH}, \cite[Lemma 4.8]{FO-CCH}
 and \cite[Corollary 4.13]{FO-CCH}, respectively.

 First we give a proof of~(\ref{lem:maximizer}).
 Set $a:=(\gamma\mid \eta)$. If $a=\infty$, by Lemma~\ref{lem:HdistD},
 we have $\ds{\gamma(t),\eta(t)}\leq D$ for all $t\in \Rp$.
 We suppose $a<\infty$. By~\cite[Lemma 4.7]{FO-CCH}, we have
 $\ds{\gamma(a), \eta(a)} \leq D_1+2k_1$. 
 Then by Proposition~\ref{prop:qgeod-ray-convex},
 \begin{align*}
  \ds{\gamma(t),\eta(t)}\leq E(D_1+2k_1) + D \quad (\forall t\in [0,a]).
 \end{align*}
 
 Next we give a proof of~(\ref{lem:ray-same-param}). 
 Let $t\in \Dom \gamma\cap \Dom \eta$ with $0\leq t \leq \max\{a,b\}$.
 First we suppose $t\leq \min\{a,b\}$. Then by Lemma~\ref{lem:t-ab},
 $\ds{\gamma(t),\eta(t)}\leq E(\lambda \tilde{\theta}(0)+k_1)+D$.
 
 Now we suppose $\min\{a,b\}<t\leq \max\{a,b\}$. Since $\gamma(a)=\eta(b)$,
 we have  $\abs{a-b}\leq \tilde{\theta}(0)$.
 Then 
 \begin{align*}
 \ds{\gamma(t),\eta(t)}\leq \ds{\gamma(t),\gamma(a)} 
  + \ds{\gamma(a),\eta(b)}
  + \ds{\eta(b),\eta(t)}\leq 2(\lambda\tilde{\theta}(0) + k_1).
 \end{align*}
\end{proof}

\begin{lemma}
\label{lem:prod-est-bigon}
 For $\gamma,\gamma'\in \bar{\calL}_O$,  $t\in \Dom \gamma$ and 
 $t'\in \Dom\gamma'$, set $\alpha:= \ds{\gamma(t),\gamma'(t')}$.
 Then we have
 \begin{align*}
  (\gamma\mid \gamma')
  \geq \frac{\max\{t,t'\}-\tilde{\theta}(\alpha)}
  {E(\alpha + \lambda\tilde{\theta}(\alpha) + k_1)}.
 \end{align*}
\end{lemma}

\begin{proof}
 By Proposition~\ref{prop:qgeod-ray-convex}
 we have
 \begin{align*}
  \abs{t-t'} \leq \tilde{\theta}(\ds{\gamma(0),\gamma'(0)}+\ds{\gamma(t),\gamma'(t')})
  =\tilde{\theta}(\alpha).
 \end{align*}
 Set $u:=\min\{t,t'\}$. We have 
 \begin{align*}
  \ds{\gamma(u),\gamma'(u)}\leq \ds{\gamma(t),\gamma'(t')} + 
  \lambda\abs{t-t'} + k_1
  \leq \alpha + \lambda\tilde{\theta}(\alpha) + k_1.
 \end{align*}
 Set $c:=(E(\alpha + \lambda\tilde{\theta}(\alpha) + k_1))^{-1}$.
 Then we have
 \begin{align*}
  \ds{\gamma(cu),\gamma'(cu)} \leq cE\ds{\gamma(u),\gamma'(u)} + D \leq D+1.
 \end{align*}
 Therefore
 \begin{align*}
  (\gamma\mid \gamma')\geq cu
  \geq \frac{\max\{t,t'\}-\tilde{\theta}(\alpha)}
  {E(\alpha + \lambda\tilde{\theta}(\alpha) + k_1)}.
 \end{align*}
\end{proof}

In the rest of this section, we suppose that $X$ is proper,
that is, all closed bounded subset is compact.
By a diagonal argument, we have the following.

\begin{lemma}[{\cite[Proposition 4.17]{FO-CCH}}]
\label{lem:Ascoli-Arzela}
 Let $\{v_n\}_{n}$ be a sequence in $X$ such that 
\begin{align*}
  \lim_{n\to \infty}\ds{O,v_n}=\infty.
\end{align*} 
Then there exists a $(\lambda,k_1)$-quasi-geodesic ray 
 $\gamma\in \calL_O^\infty$ starting at $O$, and a 
 sequence $(N_n)$ in $\N$
 such that $\liminf_{n\to \infty}(v_{N_n}\mid [\gamma])=\infty$.
\end{lemma}

\begin{corollary}
 \label{cor:As-Arz}
 Let $(x_n)_n$ be a sequence in $X$ such that $\lim_{n}\ds{O,x_n}= \infty$. Then there exists a subsequence
 $(x_{N_n})_n$ such that $\lim_{k,l}(x_{N_k}\mid x_{N_l})=\infty$.
\end{corollary}

\begin{proof}
 By Lemma~\ref{lem:Ascoli-Arzela}, there exists $\gamma\in \calL_O^\infty$ and
 sequence $(N_n)$ in $\N$ such that $\liminf_{n\to \infty}(x_{N_n}\mid [\gamma])=\infty$. 
 So by Lemma~\ref{lem:univ-const} (\ref{lem:qrayultmvisu}), we have.
 \begin{align*}
  (x_{N_k}\mid x_{N_l})\geq \DA^{-1}\min\{(x_{N_k}\mid [\gamma]), ([\gamma]\mid x_{N_l})\}\to \infty.
 \end{align*}
\end{proof}

By \cite[Corollary 4.21]{FO-CCH}, the inclusion 
$\calL_O^\infty \hookrightarrow \calL^\infty$ 
induces the bijection $\partial_O X\rightarrow \partial X$. 
Thus we identify $\partial_O X$ with $\partial X$.

\subsection{Topology on the boundary}
We define a topology on $X\cup \partial_O X$ as follows. For $n\in \N$, set
\begin{align*}
 V_n:=\{(x, y)\in (X\cup \partial_O X)^2:(x\mid y)>n\}\cup \{(x,y)\in X^2:\ds{x,y}<n^{-1}\}.
\end{align*}
The family $\{V_n\}_{n\in\N}$ forms an uniform structure on $X\cup \partial_O X$, which is metrizable.
For $x\in X\cup \partial_O X$, set $V_n(x):=\{y\in X\cup \partial_O X:(x,y)\in V_n\}$. 
Then $\{V_n(x)\}_n$ is a fundamental neighbourhood of
$x$ and the inclusion $X\hookrightarrow X\cup \partial_O X$ is a topological embedding.
Lemma~\ref{lem:Ascoli-Arzela} implies the following.
\begin{proposition}[{\cite[Proposition 4.18]{FO-CCH}}]
 Let $X$ be a coarsely convex space. If $X$ is proper, then $X\cup \partial_O X$ is compact.
\end{proposition}

\subsection{Metric on the boundary}
\begin{proposition}
\label{prop:q-met2met}
 For sufficiently small $\epsilon>0$, 
 there exists a metric $\dbdl{\epsilon}{\partial X}$ on $\partial X$ and a constant $0<K<1$ 
 such that
\begin{align*}
  \frac{1}{K}(x\mid y)^{-\epsilon} \leq 
 \dbdl{\epsilon}{\partial X}(x,y)\leq (x\mid y)^{-\epsilon}
\end{align*} 
for all $x,y\in \partial_O X$.
\end{proposition}
We refer \cite[Section 4.5]{FO-CCH} for details.

\subsection{Sequential boundary}

Let $(x_n)_n$ be a sequence in $X$. 
We say that $(x_n)_{n}$ tends to infinity if 
\begin{align*}
 \lim_{n,m\to \infty}(x_n\mid x_m)=\infty.
\end{align*}


\begin{lemma}
\label{lem:tendstoinfty}
 Suppose that $X$ is proper. 
 Then for a sequence $(x_n)_n$ in $X$ tending to infinity, 
 there exists $x\in \partial X$ such that $x_n\to x$. 
\end{lemma}

\begin{proof}
 By Lemma~\ref{lem:Ascoli-Arzela}, there exists $\gamma\in \calL_O$ and 
 $(N_n)_n\subset \N$ such that 
 $\lim_{n\to \infty}(x_{N_n}\mid [\gamma]) = \infty$. 
 So set $x:=[\gamma]$ and we have $x_{N_n}\to x$. Since for 
 all $m,n\in \N$
 \begin{align*}
  (x_m\mid x)\geq \DA^{-1}\min\{(x_m\mid x_{N_n}), \, (x_{N_n}\mid x)\},
 \end{align*}
 we have $(x_m\mid x)\to \infty$. 
\end{proof}

\begin{lemma}
\label{lem:associate}
 For $\gamma\in \calL_O^\infty$ and a sequence 
 $(x_n)_{n}\subset X$ satisfying 
 \begin{align*}
  \limsup_{n\rightarrow \infty}\ds{x_n,\Im \gamma} <\infty,
 \end{align*} 
 we have $\lim(x_n\mid [\gamma]) = \infty$.
\end{lemma}

\begin{proof}
 Set $S:=\limsup_{n\rightarrow \infty}\ds{x_n,\Im \gamma} <\infty$. 
 There exists $N$
 such that for all $n>N$, $\ds{x_n,\Im \gamma} \leq S$.
 We choose $\gamma_n\in \calL_O$ joining $O$ and $x_n$. 
 Set $[0,a_n]= \Dom \gamma_n$. For each $n>N$, there exists 
 $t_n\in \Rp$ such that $\ds{\gamma_n(a_n),\gamma(t_n)}<S+1$. Thus
 By Lemma~\ref{lem:prod-est-bigon}, we have
 \begin{align*}
  (x_n\mid [\gamma]) \geq (\gamma_n\mid \gamma)\geq 
   \frac{a_n-\tilde{\theta}(S+1)}
  {E(S+1 + \lambda\tilde{\theta}(S+1) + k_1)}.
 \end{align*}
 Since $a_n\to \infty$, we have $(x_n\mid [\gamma]) \to \infty.$
\end{proof}



\section{Visual map and radial map}
\label{sec:visual-map-radial}
Throughout this section, let $X$ be a proper 
$(\lambda, k, E, C,\theta, \xcalL{X})$-coarsely convex space, and 
let $Y$ be a proper
$(\lambda', k', E', C',\theta', \xcalL{Y})$-coarsely convex space.
Let $a\in X$ and $b\in Y$ be base points of $X$ and $Y$, respectively.
Moreover, let $\DA$ be a constant satisfying the statements of 
Lemma~\ref{lem:univ-const} for both $X$ and $Y$.

\subsection{Visual map}

\begin{definition}
 We say that a large scale Lipschitz map $f\colon X\to Y$ is \textit{visual}
 if for every pair of sequences $(x_n)_{n}, (y_n)_{n}$ in $X$ with 
 $(x_n\mid y_n)_{a} \to \infty$, we have $(f(x_n)\mid f(y_n))_b\to \infty$.
\end{definition}



\begin{lemma}
\label{lem:visualmap-r-s}
 Let $f\colon X\to Y$ be a large scale Lipschitz map. $f$ is visual if 
 and only if for all $r>0$ there exists $s>0$ such that for $x,y\in X$ with
 $(x\mid y)_a> s$, we have $(f(x)\mid f(y))_b>r$.
\end{lemma}

\begin{proof}
 Suppose that there exists $r>0$ such that for all $s>0$, there exists
 $x,y\in X$ satisfying  $(x\mid y)_a>s$ and $(f(x)\mid f(y))_b\leq r$.
 Then we can find sequences $(x_n)_{n\in \N}$ and $(y_n)_{n\in \N}$ such
 that
 \begin{align*}
  (x_n\mid y_n)_a > n,\quad (f(x_n)\mid f(y_n))_b \leq r 
  \quad (\forall n\in \N).
 \end{align*}
 Then $(x_n\mid y_n)_a\to \infty$ and 
 $\limsup_n(f(x_n)\mid f(y_n))_b\leq r$. So $f$ is not visual.
 The converse is clear.
\end{proof}

\begin{proposition}
 \label{prop:visual-proper}
 Let $f\colon X\to Y$ be a visual map. Then $f$ is metrically proper.
\end{proposition}

\begin{proof}
 We prove the contraposition. Suppose $f$ is not metrically proper. Then 
 there exists a sequence $(x_n)_{n}$ in $X$ 
 such that $\ds{x_n,a}\to \infty$ and $\sup_{n} \ds{f(x_n),b}<\infty$. 
 By Corollary~\ref{cor:As-Arz} there exists $(N_n)_n\subset \N$ such that 
 $(x_{N_n}\mid x_{N_n}) \to \infty$. 

 For each $n\in \N$, we choose $\gamma_n\in \xcalL{Y}_b$ with 
 $\Dom \gamma_n = [0,L_n]$
 such that $\gamma_n(L_n) = f(x_{N_n})$. 
 By Lemma~\ref{lem:univ-const} (\ref{lem:D3}) we have
 \begin{align*}
  (f(x_{N_n})\mid f(x_{N_n}))_b &\leq \DA(\gamma_n\mid \gamma_n)_b
  \leq \DA L_n \\
  &\leq \DA(\lambda'(\ds{f(x_{N_n}),b}+k')) 
  \leq \lambda'\DA(\sup_{n} \ds{f(x_n),b}+k')<\infty.
 \end{align*}
 Therefore $f$ is not visual.
\end{proof}

\begin{definition}
 Let $f\colon X\to Y$ be a visual map. We define a map
 $\partial f\colon \partial X \to \partial Y$ as follows. 
 For $x \in \partial X$, we choose a representative 
 $\gamma \in \xcalL{X}_a^\infty$ of $x$, that is, $\gamma\in x$. 
 Then we define $\partial f(x):= \lim f(\gamma(n))$. The following proposition says that
 the limit does exists.
 We say that $\partial f$ is \textit{induced} by $f$.
\end{definition}

\begin{proposition}
\label{prop:visual-extension}
 Let $f\colon X\to Y$ be a visual map.
 The map $\partial f$ is well-defined,
 and for $x\in \partial X$, $\partial f(x)$ does not depend on the choice
 of the representative $\gamma\in x$.
\end{proposition}

\begin{proof}
 For $x\in \partial X$, we choose a representative 
 $\gamma \in \xcalL{X}_a^\infty$ of $x$, that is, $\gamma\in x$. 
 For $n\in \N$, set $x_n:=\gamma(n)$, $y_n:=f(x_n)$. 
 By Lemma~\ref{lem:associate}, we have $(x_n\mid x_m)_a\to \infty$.
 Since $f$ is visual , we have 
 $(y_n\mid y_m)_b\to \infty$. Then by Lemma~\ref{lem:tendstoinfty}, 
 there exists $y\in \partial Y$ such that $y_n\to y$. 
 Therefore $\partial f(x) = y$ is well-defined.

 We will show that $y$ does not depend on the choice of $\gamma$.
 Let $\eta\in x$ and set $x'_n:=\eta(n)$, $y'_n:=f(x'_n)$, 
 $y:=\lim_{n\to \infty}y'_n$. Since
 \begin{align*}
  (x_n\mid x'_n)_a \geq \DA^{-1}\min\{(x_n\mid x)_a,\, (x\mid x'_n)_a\}
  \to \infty,
 \end{align*}
 we have  $(y_n\mid y'_n)_{b}\to \infty$. So
 \begin{align*}
  (y\mid y')_b &\geq \DA^{-1}\min\{(y\mid y_n)_b,\,(y_n\mid y')_b\}\\
  &\geq \DA^{-2}\min\{(y\mid y_n)_{b},\, (y_n\mid y'_n)_{b},\, (y'_n\mid y')_{b}\}
  \to \infty.
 \end{align*}
 Therefore $y=y'$. 
\end{proof}

 Lemma~\ref{lem:visualmap-r-s} and following Lemma~\ref{lem:partialf-r-s}
 imply that $\partial f$ is continuous on $\partial X$.

\begin{lemma}
\label{lem:partialf-r-s}
 Let $f\colon X\to Y$ be a visual map and 
 $\partial f\colon \partial X\to \partial Y$ be a map induced by $f$.
 For $r>0$ there exists $s>0$ such that for $x,y\in \partial X$ with
 $(x\mid y)>s$, we have $(\partial f(x)\mid \partial f(y))>r$.
\end{lemma}

\begin{proof}
 We fix $r>0$. By Lemma~\ref{lem:visualmap-r-s},
 there exists $s'>0$ such that for all $p,q\in X$ 
 \begin{align*}
  (p\mid q)_a > s' \Rightarrow (f(p)\mid f(q))_{b}> \DA^2 r.
 \end{align*}
 Set $s:= \DA^{2} s'$. Let $x,y\in \partial X$ with $(x\mid y)_{a} > s$.
 We choose representative $\gamma\in x$ and $\eta\in y$. For $n\in \N$,
 set $x_n:=\gamma(n)$ and $y_n:=\eta(n)$. 
 Since
 \begin{align*}
  (x_n\mid y_n)_{a}\geq \DA^{-2}\min \{(x_n\mid x)_{a}, (x\mid y)_{a}, (y\mid y_n)_{a}\}
 \end{align*}
 and $(x_n\mid x)_{a}\to \infty$ , $(y\mid y_n)_{a}\to \infty$, 
 for sufficiently large $n$, we have $(x_n\mid y_n)_{a}>s'$. 
 Then $(f(x_n)\mid f(y_n))_{b}>\DA^2 r$. Now we have
 \begin{align*}
  (\partial f(x)\mid \partial f(y))_{b}\geq \DA^{-2}
  \min \{(\partial f(x)\mid f(x_n))_{b},\, 
  (f(x_n)\mid f(y_n))_{b},\, (f(y_n)\mid \partial f(y))_{b}\}.
 \end{align*}
 Since $(\partial f(x)\mid f(x_n))_{b}\to \infty$ and 
 $(f(y_n)\mid \partial f(y))_{b}\to \infty$,
 we have $(\partial f(x)\mid \partial f(y))_{b}> r$.
\end{proof}

\begin{corollary}
 \label{cor:visu-bdl-cont}
 Let $f\colon X\to Y$ be a visual map.
 Then the map 
 \begin{align*}
  f\cup \partial f\colon X\cup \partial X \to Y \cup \partial Y
 \end{align*} 
 is continuous at every point in $\partial X$.
\end{corollary}

\begin{proof}
 We fix $x\in \partial X$. We will show that for $r>0$, there exists $s>0$
 such that, for $y\in X\cup \partial X$,
 \begin{align*}
  (x\mid y)_a> \DA s \Rightarrow 
  (\partial f(x)\mid (f\cup \partial f)(y))_b> \DA^{-1}r.
 \end{align*}

 For $r>0$, we choose $s>0$ satisfying the statements of both 
 Lemma~\ref{lem:visualmap-r-s} and Lemma~\ref{lem:partialf-r-s} for $f$ and $r$.

 First let $y \in \partial X$ with $(x\mid y)_a>s$.
 Then by Lemma~\ref{lem:partialf-r-s}, we have
 $(\partial f(x)\mid \partial f(y))_b>r$.

 Now let $y\in X$ with $(x\mid y)_a> \DA s$.
 Let $\gamma\in \xcalL{X}_a^\infty$ with $\gamma\in x$.
 By the proof of Proposition~\ref{prop:visual-extension},
 there exists $n\in \N$ such that $(x\mid \gamma(n))_a>\DA s$ and
 $(\partial f(x)\mid f(\gamma(n)))_b > r$. 
 Then
 \begin{align*}
  (y\mid \gamma(n))_a\geq \DA^{-1} \min\{(y\mid x)_a,(x\mid \gamma(n))_a\}
  > s.
 \end{align*}
 Thus by Lemma~\ref{lem:visualmap-r-s}, 
 $(f(\gamma(n))\mid f(y))_b> r$. Therefore
 \begin{align*}
  (\partial f(x)\mid f(y))_b\geq 
  \DA^{-1}\min\{(\partial f(x)\mid f(\gamma(n))_b,(f(\gamma(n))\mid f(y))_b\}
  >\DA^{-1}r.
 \end{align*}
\end{proof}

\subsection{Maps compatible with systems of good quasi-geodesics}
Let $\tau\geq 0$ be a constant.
We say that a map $\sigma\colon \Rp\to\Rp$ is a $\tau$-\textit{rough contraction}
if $\sigma(0)=0$, $\lim_{t\to\infty}\sigma(t)=\infty$, and
\begin{align*}
 \rho(t)\leq t \quad (\forall t\in \Rp),\\
 \abs{\rho(t) - \rho(s)}\leq \abs{t-s} + \tau  \quad(\forall t,s\in \Rp),\\
 t\leq s \Rightarrow \rho(t)\leq \rho(s) \quad (\forall t,s\in\Rp,).
\end{align*}
We say that $\sigma$ is a \textit{rough contraction} if it is $\tau$-rough contraction for some
$\tau\geq 0$.

Dydak and Virk~\cite[Definition 3.2]{dydak-G-bdry-maps} introduced a class of radial function, 
in the setting of Gromov hyperbolic spaces. The following is its analogue in the setting of
coarsely convex spaces. We remark that the following is slightly weaker than the one defined by
Dydak and Virk.

\begin{definition}
 Let $\sigma\colon \Rp\to \Rp$ be a rough contraction.
 We say that a map $f\colon X\to Y$ 
 is $\sigma$-$\xcalL{X}$-\textit{radial}
 if for $\gamma\in \xcalL{X}_a$, we have
 \begin{align*}
  \sigma(t) \leq \ds{f(\gamma(t)),f(a)} 
  \quad (\forall t\in \Dom\gamma).
 \end{align*}
 We say that $f$ is \textit{$\xcalL{X}$-radial} if it is
 $\sigma$-$\xcalL{X}_a$-radial for some $\sigma$ and $a\in X$.
\end{definition}

The above condition is in fact equivalent to the following. 
Let $\sigma\colon \Rp\to\Rp$ be a rough contraction. 
We say that a map
$f\colon X\to Y$ is \textit{weakly $\sigma$-$\xcalL{X}_a$-radial} if 
for all $x\in X$, there exists $\gamma_x\in \xcalL{X}_a$ and 
$T_x\in \Rp$ such that $x=\gamma_x(T_x)$ and 
 \begin{align*}
  \sigma(T_x) \leq \ds{f(x),f(a)}.
 \end{align*}

\begin{lemma}
\label{lem:weak-rad-rad}
 Let $\sigma\colon \Rp\to\Rp$ be a rough contraction. 
 If a large scale Lipschitz map $f\colon X\to Y$
 is \textit{weakly $\sigma$-$\xcalL{X}_a$-radial},
 then it is $\hat{\sigma}$-$\xcalL{X}_a$-radial for some 
 rough contraction $\hat{\sigma}$.
\end{lemma}

\begin{proof}
 Let $\gamma\in \xcalL{X}_a$. 
 We fix $t\in \Dom\gamma$ and set  
 $x=\gamma(t)$. 
 Since $f$ is weakly $\sigma$-$\xcalL{X}_a$-radial, there exists 
 $\gamma_x\in \xcalL{X}_a$ and $T_x\in \Rp$ such that 
 $x = \gamma_x(T_x)$ and 
 \begin{align*}
  \sigma(T_x) \leq \ds{f(\gamma_x(t)),f(a)}.
 \end{align*}
 Since $\gamma(t)=x=\gamma_x(T_x)$, we have $\abs{t-T_x}\leq \theta(0)$.
 We suppose that $\sigma$ is $\tau$-rough contraction. Then we have
 \begin{align*}
  \abs{\sigma(T_x)-\sigma(t)}\leq \abs{T_x-t}+\tau\leq \theta(0)+\tau.
 \end{align*}
 Thus we have
 \begin{align*}
  \ds{f(\gamma(t)),f(a)} 
  &= \ds{f(\gamma_x(T_x)),f(a)} \\
  &\geq \sigma(T_x)\\
  &\geq \sigma(t) -\theta(0)-\tau.
 \end{align*}
 Set 
 $\hat{\sigma}(t):=\left(\sigma(t) -\theta(0)-\tau\right)\vee 0$. Then $f$ is $\hat{\sigma}$-$\xcalL{X}_a$-radial.
\end{proof}

In the study of maps between Gromov hyperbolic spaces, the Morse lemma plays 
essential roles. In the case of coarsely convex spaces, in general, the Morse
lemma does not holds. Instead, we introduce the following.

\begin{definition}
\label{def:LL-equiv}
 Let $\rho\colon \Rp\to \Rp$ be  a rough contraction. 
 Let $H>0$. We say that a map $f\colon X\to Y$ is 
 $\rho$-$H$-$\xcalL{X}_a$-$\xcalL{Y}_b$-\textit{equivariant} if
 $f(a)=b$, and, 
 for $\gamma\in \xcalL{X}_a$ with $\Dom \gamma = [0,L_\gamma]$ 
 and for $\eta\in \xcalL{Y}_b$ with 
 $\Dom \eta = [0,L_\eta]$ satisfying
 $\eta(L_\eta)=f(\gamma(L_\gamma))$, we have
 \begin{align*}
  \ds{f\circ \gamma(t),\eta(\rho(t))}< H 
  \quad (0\leq \forall t\leq \min\{L_\gamma,\rho^{-1}(L_\eta)\}).
 \end{align*}
 We say that $f$ is \textit{$\xcalL{X}$-$\xcalL{Y}$-equivariant} if it is
 $\rho$-$H$-$\xcalL{X}_a$-$\xcalL{Y}_b$-equivariant for some 
 $\rho$, $H$, $a\in X$ and $b\in Y$.
\end{definition}

\begin{theorem}
\label{thm:visual-Lip}
 For $\lambda_1\geq 1$, $\nu_1,H,\tau> 0$,
 and $\tau$-rough contractions
 $\sigma,\rho\colon \Rp\to \Rp$, there exists
 a metrically proper map $\hat{\rho}\colon \Rp\to \Rp$
 such that the following holds;
 let $f\colon X\to Y$ be a 
 $(\lambda_1,\nu_1)$-large scale Lipschitz, 
 $\sigma$-$\xcalL{X}_a$-radial, 
 and $\rho$-$H$-$\xcalL{X}_a$-$\xcalL{Y}_b$-equivariant map satisfying $f(a) = b$.
 Then for $x,y\in X$, we have
 \begin{align*}
  (f(x)\mid f(y))_{b} \geq \hat{\rho}((x\mid y)_a).
 \end{align*}
 Especially, $f$ is visual. 
\end{theorem}

\begin{proof}
 Let $\gamma,\eta\in \xcalL{X}_a$, be
 a good quasi-geodesics joining the base point $a$ and $x,y$, 
 respectively.
 Set $[0,L_\gamma]=\Dom \gamma$, $[0,L_{\eta}]=\Dom \eta$.
 We remark that $(\gamma\mid \eta)_a\geq \DA^{-1}(x\mid y)_a$.

 Let $\gamma',\eta'\in \xcalL{Y}_b$  be
 a good quasi-geodesics joining the base point $b$ and $f(x),f(y)$,
 respectively.
 Set $[0,L_{\gamma'}]=\Dom \gamma'$, $[0,L_{\eta'}]=\Dom \eta'$.
 We have
 \begin{align*}
  \min\{L_\gamma,L_\eta\} \geq (\gamma\mid \eta)_a\geq \DA^{-1}(x\mid y)_a.
 \end{align*}
 Since $\gamma'$ is $(\lambda',k')$-quasi geodesic and 
 $f$ is $\sigma$-$\xcalL{X}_a$-radial, it follows that
 \begin{align*}
  \lambda'L_{\gamma'} + k'&\geq \ds{b,f(x)} \geq \sigma(L_\gamma)
  \geq \sigma(\DA^{-1}(x\mid y)).
 \end{align*}
 By switching $x$ and $y$, and applying the same argument, we obtain
 \begin{align*}
  \lambda'L_{\eta'} + k' \geq \sigma(\DA^{-1}(x\mid y)_a).
   \end{align*}
 So we have
 \begin{align}
  \label{eq:3}
  \min\{L_{\gamma'},L_{\eta'}\}\geq 
  \frac{1}{\lambda'}\sigma(\DA^{-1}(x\mid y)_a) -\frac{k'}{\lambda'}.
 \end{align} 

 Here we introduce the following constants,
 \begin{align*}
  T:=\lambda_1\DA + \nu_1 +2H,\quad c:=(E'T)^{-1}
 \end{align*}
 and we define a function $\delta\colon \Rp\to \Rp$ by
 \begin{align*}
  \delta(t):=\sup\left\{u\in\Rp: \rho(u)\leq   
  \left(\frac{1}{\lambda'}\min\{\sigma(\DA^{-1}t),\,\rho(\DA^{-1}t)\}  
  -\frac{k'}{\lambda'}-\tau \right)\vee 0\right\}.
 \end{align*}
 We remark that, since $\rho$ is a  $\tau$-rough contraction, 
 for any $\alpha\in \Rp$, we have
 \begin{align*}
  \rho(\sup\{u\in \Rp: \rho(u)\leq \alpha\}) \leq \alpha +\tau.
 \end{align*}

 We will show that 
 \begin{align}
  \label{eq:8}
  (f(x)\mid f(y))\geq c\left\{\rho\circ \delta((x\mid y)_a) -\tau\right\}.
 \end{align}
 We set 
 \begin{align*}
  S&:= \frac{1}{\lambda'}
  \min\{\sigma(\DA^{-1}(x\mid y)_a),\,\rho(\DA^{-1}(x\mid y)_a)\}
  -\frac{k'}{\lambda'},\\
 s&:=\delta((x\mid y)_a) = \sup\{u \in\Rp : \rho(u)\leq (S-\tau)\vee 0\}.
 \end{align*}
 Then we have $\rho(s)\leq S\vee \tau$ and 
 $\rho\circ \delta((x\mid y)_a)=\rho(s)$.

 First we assume that $S<\tau$. Then $\rho\circ \delta((x\mid y)_a) \leq \tau$,
 so clearly (\ref{eq:8}) holds, since $(f(x)\mid f(y))\geq 0$.
 
 Now we assume that $S\geq \tau$. 
 Since $\rho(s)\leq S$ and $\lambda'\geq 1$, we have 
 \begin{align*}
  \rho(s)\leq \rho(\DA^{-1}(x\mid y)_a)\leq \rho((\gamma\mid \eta)_a). 
 \end{align*}
 Thus
 \begin{align*}
  s\leq (\gamma\mid \eta)_a\leq \min\{L_\gamma,L_\eta\}.
 \end{align*}
 By~(\ref{eq:3}) and $\rho(s)\leq S$, we have 
 \begin{align*}
  0\leq \rho(s)\leq \min\{L_{\gamma'},L_{\eta'}\}.
 \end{align*}
 Since $f$ is 
 $\rho$-$H$-$\xcalL{X}_a$-$\xcalL{Y}_b$-equivariant, 
 we have
 \begin{align*}
 \max\left\{\ds{\gamma'(\rho(s)),f\circ \gamma(s)}, 
  \ds{f\circ \eta(s),\eta'(\rho(s))}\right\}<H.
 \end{align*}
 By Lemma~\ref{lem:univ-const}~(\ref{lem:maximizer})
 we have
 \begin{align*}
  \ds{\gamma(s),\eta(s)}\leq \DA.
 \end{align*}
Thus,
 \begin{align*}
  \ds{\gamma'(\rho(s)),\eta'(\rho(s))}
  &\leq \ds{\gamma'(\rho(s)),f\circ \gamma(s)} +
  \ds{f\circ \gamma(s),f\circ \eta(s)} + \ds{f\circ \eta(s),\eta'(\rho(s))}\\
  &\leq H + \lambda_1\ds{\gamma(s),\eta(s)} +\nu_1 + H\\
  &\leq \lambda_1\DA + \nu_1 +2H=T.
 \end{align*}
 Since $c=1/(E'T)$, we have
 \begin{align*}
  \ds{\gamma'(c\rho(s)),\eta'(c\rho(s))}\leq 1 +C'
  \leq 2D' + 2.
 \end{align*}
 Thus we have 
 \begin{align*}
  (\gamma'\mid \eta')_b\geq c\rho(s) = c\rho\circ \delta((x\mid y)_a).
 \end{align*}
 Therefore (\ref{eq:8}) holds.
\end{proof}

\begin{lemma}
\label{lem:partial-f-kapp}
Let $f\colon X\to Y$ be a map.
 Suppose that there exists a non-decreasing map 
 $\hat{\rho}\colon \Rp\to \Rp$ such that for 
 all $x,y\in X$, $(f(x)\mid f(y))_b\geq \hat{\rho}((x\mid y)_a)$.
 Then for all $u,v\in \partial X$,
 we have 
 \begin{align*}
  (\partial f(u)\mid \partial f(v))_b\geq 
  \DA^{-2}\hat{\rho}(\DA^{-2}(u\mid v)_a).
 \end{align*}
\end{lemma}

\begin{proof}
 Let $u,v\in \partial_a X$.
 We choose representatives $\gamma\in u$ and $\eta\in v$. For $n\in \N$,
 set $x_n:=\gamma(n)$ and $y_n:=\eta(n)$. 
 We have
 \begin{align*}
  (x_n\mid y_n)_{a}\geq \DA^{-2}\min \{(x_n\mid u)_{a}, (u\mid v)_{a}, (v\mid y_n)_{a}\}
 \end{align*}
 Since $(x_n\mid u)_{a}\to \infty$ and $(v\mid y_n)_{a}\to \infty$,
 we have 
 $(x_n\mid y_n)_{a}\geq \DA^{-2}(u\mid v)_{a}$
 for sufficiently large $n$.
 Then $(f(x_n)\mid f(y_n))_{b}\geq \hat{\rho}(\DA^{-2}(u\mid v)_a)$. 
 Now we have
 \begin{align*}
  (\partial f(u)\mid \partial f(v))_{b}\geq \DA^{-2}
  \min \{(\partial f(u)\mid f(x_n))_{b},\, 
  (f(x_n)\mid f(y_n))_{b},\, (f(y_n)\mid \partial f(v))_{b}\}.
 \end{align*}
 Since $(\partial f(u)\mid f(x_n))_{b}\to \infty$ and 
 $(f(y_n)\mid \partial f(v))_{b}\to \infty$,
 we have 
 \begin{align*}
 (\partial f(u)\mid \partial f(v))_{b}> 
  \DA^{-2} \hat{\rho}(\DA^{-2}(u\mid v)_{a}).
 \end{align*}
\end{proof}

We say that a map $\sigma\colon \Rp\to\Rp$ is a \textit{rough-isometric contraction}
if it is a rough contraction and there exist  constants $A,B\geq 0$ such that
$\abs{\sigma(t)-\sigma(s)}\geq A\abs{t-s} -B$ for all $s,t\in \Rp$.

\begin{corollary}
\label{cor:visu-Lip}
 Let $\sigma,\rho\colon \R\to \R$ be rough-isometric contractions. 
 Let $f\colon X\to Y$ be a 
 large scale Lipschitz, $\sigma$-$\xcalL{X}_a$-radial,  
 and $\rho$-$H$-$\xcalL{X}_a$-$\xcalL{Y}_b$ 
 map satisfying $f(a) = b$. Then the induced map
 $\partial f\colon \partial X\to \partial Y$ is a Lipschitz map.
\end{corollary}

\begin{proof}
 Let $\epsilon>0$ be a sufficiently small constant and,
 $\dbdl{\epsilon}{\partial X}$ and $\dbdl{\epsilon}{\partial Y}$ be 
 a metric on $\partial X$ and on $\partial Y$, respectively, constructed 
 in Proposition~\ref{prop:q-met2met}. Let $K$ be a constant 
 satisfying the statement of 
 Proposition~\ref{prop:q-met2met} for both $\partial X$ and $\partial Y$.

 Since $\sigma$ and $\rho$ are rough-isometric contractions,
 by the proof of Theorem~\ref{thm:visual-Lip} and 
 Lemma~\ref{lem:partial-f-kapp},
 there exists constants $\kappa,\nu$ such that for all
 $x,y\in \partial X$, we have 
 \begin{align*}
  (\partial f(x)\mid \partial f(y))_b\geq \kappa (x\mid y)_a -\nu.
 \end{align*}

 We use the following constant
 \begin{align*}
 \displaystyle V:= \frac{1}{K}
  \left(\frac{\kappa}{2\nu}\right)^{\epsilon}. 
 \end{align*}
 For $x,y\in \partial X$ with 
 $\dbdl{\epsilon}{\partial X}(x,y)\geq V$, we have
 \begin{align*}
  \dbdl{\epsilon}{\partial Y}(\partial f(x),\partial f(y))\leq \diam \partial Y\leq 
  \frac{\diam \partial Y}{V}\dbdl{\epsilon}{\partial X}(x,y).
 \end{align*}
 For $x,y\in \partial X$ with  $\dbdl{\epsilon}{\partial X}(x,y)\leq  V$, 
 we have
 \begin{align*}
  (x\mid y)_a \geq \left(\frac{1}{K\dbdl{\epsilon}{\partial X}(x,y)}\right)^{1/\epsilon}
  \geq \left(\frac{1}{KV}\right)^{1/\epsilon}
  \geq \frac{2\nu}{\kappa}.
 \end{align*}
 Thus $\nu\leq (1/2)\kappa (x\mid y)_b$. 
 By Theorem~\ref{thm:visual-Lip},
 \begin{align*}
  \dbdl{\epsilon}{\partial Y}(\partial f(x),\partial f(x')) &\leq 
  \left(\frac{1}{(f(x)\mid f(y))_b}\right)^\epsilon 
  \leq   \left(\frac{1}{\kappa (x\mid y)_a -\nu}\right)^\epsilon \\
  &\leq \left(\frac{2}{\kappa (x\mid y)_a}\right)^\epsilon 
  \leq K\left(\frac{2}{\kappa}\right)^\epsilon \dbdl{\epsilon}{\partial X}(x,x').
 \end{align*}
\end{proof}

\begin{definition}
\label{def:LLoo-radial}
 Let $\rho\colon \Rp\to \Rp$ be a rough contraction.
 Let $f\colon X\to Y$ be a map.
 Let $T>0$. We say that $f$ is 
 $\rho$-$T$-$\xcalL{X}_a^\infty$-$\xcalL{Y}_b^\infty$-\textit{equivariant} if $f(a) = b$, and
 for $\gamma\in \xcalL{X}_a^\infty$
 there exists $\eta\in \xcalL{Y}_b^\infty$ such that
 \begin{align*}
  \ds{f\circ \gamma(t),\eta(\rho(t))}< T \quad (\forall t\in \Rp).
 \end{align*}
\end{definition}

\begin{lemma}
 \label{lem:LLoo-f} 
 Let $\rho\colon \Rp\to \Rp$ be a rough contraction.
 Let $f\colon X\to Y$ be a visual map.
 Let $\gamma\in \xcalL{X}_a^\infty$ and $\eta\in \xcalL{Y}_b^\infty$
 be quasi-geodesic rays
 satisfying $\ds{f\circ \gamma(t),\eta(\rho(t))}< T$ 
 for all $t\in \Rp$.
 Then we have $\partial f([\gamma]) = [\eta]$.
\end{lemma}

\begin{proof}
 For $n\in \N$, set $x_n:=\gamma(n)$. By Lemma~\ref{lem:associate},
 $x_n\to [\gamma]$.
 By Corollary~\ref{cor:visu-bdl-cont}, $f(x_{n})\to \partial f([\gamma])$,
 that is,  $(f(x_{n})\mid \partial f([\gamma]))_b\to \infty$. 
 By the assumption of $\gamma$ and $\eta$, for 
 all $n\in \N$, $\ds{f(x_n),\eta(\rho(n))}<T$. 
 Thus by Lemma~\ref{lem:associate}, 
 $(f(x_{n})\mid [\eta])_b\to \infty$. Then
 \begin{align*}
  (\partial f([\gamma])\mid [\eta])_b\geq 
  \DA^{-1}\min\{(\partial f([\gamma])\mid f(x_{n}))_b, 
  (f(x_{n})\mid [\eta])_b\}\to \infty.
 \end{align*}
 Therefore $\partial f([\gamma])=[\eta]$.
\end{proof}

\begin{lemma}
 \label{lem:LL-LLoo}
 For $H>0$ there exists 
 $T>0$ such that if a large scale Lipschitz map $f\colon X\to Y$ is
 $\xcalL{X}$-radial and 
 $\rho$-$H$-$\xcalL{X}_a$-$\xcalL{Y}_b$-equivariant, then it
 is $\rho$-$T$-$\xcalL{X}_a^\infty$-$\xcalL{Y}_b^\infty$-equivariant.
\end{lemma}

\begin{proof}
 Let $\gamma\in \xcalL{X}_a^\infty$. There exists $\xcalL{X}$-approximate sequence 
 $\{(\gamma_n,\alpha'_n)\}$ of $\gamma$. Set $\alpha_n:=(\gamma\mid \gamma_n)_a$. 
 By Lemma~\ref{lem:univ-const}~(\ref{lem:maximizer})
 for all $t\in [0,\alpha_n]$, we have
 \begin{align*}
  \ds{\gamma(t),\gamma_n(t)}\leq \DA.
 \end{align*}
 We suppose that  $f$ is a $(\lambda_1,\nu_1)$-large scale Lipschitz map. Then
 \begin{align}
  \label{eq:2}
  \ds{f(\gamma(t)),f(\gamma_n(t))}\leq \lambda_1\DA + \nu_1.
 \end{align}
 Set $x_n:=\gamma_n(\alpha_n')$ and $y_n:=f(x_n)$. Let $\eta_n\in \xcalL{Y}_b$ be a good quasi-geodesic from $b$ to $y_n$.
 Since $f$ is $\xcalL{X}$-radial, $\ds{b,y_n}\to \infty$. Then there exists $\eta\in \xcalL{Y}_b^\infty$ and $(N_n)_N\subset \N$ 
 such that ${\eta_{N_n}}$ is a 
 $\xcalL{Y}$-approximate sequence of $\eta$. 
 Set $\beta_n:=(\eta_{N_n}\mid \eta)_b$. 

 Now we fix $t\in \Rp$. By Lemma~\ref{lem:prod-approx}, 
 there exists $n$ such that 
 $\rho(t)\leq t \leq \min\{\alpha_{N_n},\beta_n\}$. 
 By Lemma~\ref{lem:univ-const}~(\ref{lem:maximizer}),
 \begin{align}
  \label{eq:10}
  \ds{\eta_{N_n}(\rho(t)),\eta(\rho(t))}\leq \DA.
 \end{align}
 By (\ref{eq:2}),~(\ref{eq:10}) and the assumption that $f$ is a 
 $\rho$-$H$-$\xcalL{X}_a$-$\xcalL{Y}_b$-equivariant map,
 \begin{align*}
  \ds{f(\gamma(t)),\eta(\rho(t))}&\leq \ds{f(\gamma(t)),f(\gamma_{N_n}(t))} + \ds{f(\gamma_{N_n}(t)),\eta_{N_n}(\rho(t))}
  + \ds{\eta_{N_n}(\rho(t)),\eta(\rho(t))}\\
  &\leq \lambda_1\DA + \nu_1 + H + \DA.
 \end{align*}
 We denote by $T$, the right most side of the above inequality. 
 Then we have
 that $f$ is $\rho$-$T$-$\xcalL{X}_a^\infty$-$\xcalL{Y}_b^\infty$-equivariant.
\end{proof}

\subsection{Proof of Theorem~\ref{thm:fixed-pt}}
\label{sec:proof-theor-refthm:f}
Let $X$ be a proper coarsely convex space. 
Set $\bar{X}:=X\cup \partial X$.
By Theorem~\ref{cor:visu-bdl-cont}, $f$ extends to the continuous map
$f\cup \partial f\colon\bar{X}\to \bar{X}$.
By \cite[Proposition 6.6]{FO-CCH}, $\bar{X}$ is a coarse compactification.
So Theorem~\ref{thm:fixed-pt} follows 
from~\cite[Corollary 1.7]{MR2804542}.

\section{Coarsely surjective}
\label{sec:coarsely-surjective}


By Proposition~\ref{prop:visual-proper}, visual maps are coarse maps,
so they induce continuous maps on Higson coronae. Austin and Virk
\cite{MR3576439} showed that coarsely surjective maps induce surjections
on Higson coronae. In \cite[Section 6.2]{FO-CCH}, it is shown that for
a proper coarsely convex space $X$, $X\cup \partial X$ is a coarse
compactification. Then by the universality of the Higson
compactification~\cite[Proposition 2.39]{MR2007488}, we have the following.

\begin{proposition}
\label{prop:csurj-surj}
 Let $X$ and $Y$ be proper coarsely convex spaces.
 Let $f\colon X\to Y$ be a visual map. If $f$ is coarsely surjective, then 
 the induced map $\partial f\colon \partial X\to \partial Y$ is surjective.
\end{proposition}

For reader's convenience, we give a proof of 
Proposition~\ref{prop:csurj-surj} without using Higson coronae.

\begin{proof}
 We use same setting for $X$ and $Y$ as those stated in the beginning of
 Section~\ref{sec:visual-map-radial}. 
 Since $f$ is coarsely surjective,
 there exists $S>0$ such that $S$-neighbourhood of $f(X)$ is equal to $Y$. Then for 
 $\eta\in \xcalL{Y}_b^\infty$ and $n\in \N$, there exists 
 $x_n\in X$ such that $\ds{\eta(n),f(x_n)}\leq S$. 
 We suppose that $f$ is a $(\lambda_1,\nu_1)$-large scale Lipschitz map.
 We have
 \begin{align*}
  \ds{x_n,a}&\geq \frac{1}{\lambda_1}(\ds{f(x_n),b}-\nu_1)\geq
  \frac{1}{\lambda_1}(\ds{\eta(n),b}-\ds{\eta(n),f(x_n)}-\nu_1)\\
  &\geq \frac{1}{\lambda_1}(\ds{\eta(n),b}-S-\nu_1)\to \infty,
 \end{align*}
 Thus there exists 
 $\gamma\in \xcalL{X}_a^\infty$ and $(N_n)_n\subset \N$ such that
 \begin{align*}
  \lim_{n\to \infty}(x_{N_n}\mid [\gamma])_a = \infty.
 \end{align*}
 Then we have
 \begin{align}
  \label{eq:1}
  (\gamma(n)\mid x_{N_n})_a\geq 
  \DA^{-1}\min\{(\gamma(n)\mid [\gamma])_a,([\gamma],x_{N_n})_a\}\to \infty.
 \end{align}
 
 We will show that $\partial f([\gamma])=[\eta]$. Set 
 $y:=\partial f([\gamma])$ and $y_n=f(\gamma(n))$. Since 
 $f\cup \partial f$ is continuous at $[\gamma]\in \partial_a X$, 
 we have $(y\mid y_n)_b\to\infty$. 
 Now we have
 \begin{align*}
  (y\mid [\eta])_b\geq 
  \DA^{-2}\min\{(y\mid y_n)_b,(y_n\mid f(x_{N_n}))_b,(f(x_{N_n})\mid [\eta])_b\}
 \end{align*}
 By Lemma~\ref{lem:associate}, we have 
 $(f(x_{N_n})\mid [\eta])_b \to \infty$.
 Since $f$ is visual, by (\ref{eq:1}), we have 
 \begin{align*}
  (y_n\mid f(x_{N_n}))_b = (f(\gamma(n))\mid f(x_{N_n}))_b\to \infty.
 \end{align*}
 Thus 
 $(y\mid [\eta])_b = \infty$, therefore 
 $\partial f([\gamma]) = y = [\eta]$.
\end{proof}

\section{Coarsely n-to-one map}
\label{sec:coarsely-n-one}


We say that a map is \textit{$n$-to-1} if its point-inverses contain at most $n$ points. 

The following proposition is a generalization 
of \cite[Theorem 7.1]{dydak-G-bdry-maps} 
in the case of Gromov hyperbolic spaces, to that of coarsely convex spaces.

\begin{proposition}
\label{prop:bdry-n-to-one}
 Let $X$ and $Y$ be proper coarsely convex spaces, with base points $a\in X$ and $b\in Y$, respectively.
 Let $f\colon X\to Y$ be a visual map, and $\partial f\colon \partial X\to \partial Y$ be the map induced by $f$.
 The following is equivalent.
 \begin{enumerate}[$($A$)$]
  \item \label{item:1}
	$\partial f$ is $n$-to-one.
  \item \label{item:2}
	For all $R>0$, there exists $S>0$ such that, for all $x_1,\dots,x_{n+1}\in X$ with
	$(x_i\mid x_j)<R\, (\forall i,j,\, i\neq j)$, there exists $1\leq k<m\leq n+1$
	satisfying $(f(x_k)\mid f(x_m))_b<S$.
 \end{enumerate}
\end{proposition}

We say that the map $f$ is \textit{visually $n$}-to-one if $f$ satisfies 
the statement~(\ref{item:2}) of Proposition~\ref{prop:bdry-n-to-one}.

\begin{proof}
 First, we suppose that (\ref{item:2}) is negative, and we will shows that (\ref{item:1}) is negative.


 Since (\ref{item:2}) is negative, there exists $R>0$ such that for all 
 $l\in \N$, there exists  $x_1^l,\dots,x_{n+1}^l\in X$ with 
 $(x_i^l\mid x_j^l)_a<R$ and $(f(x_i^l)\mid f(x_j^l))_b>l$ for all 
 $i,j\in \{1,\dots,n+1\}$ 
 with $i\neq j$. Since $X$ is proper and by Proposition~\ref{cor:As-Arz}, 
 by replacing subsequence, we can assume that for each $i$, 
 the sequence $(x_i^l)_l$ converges to some $x_i\in \partial X$ when 
 $l\to \infty$. 
 Since $(x_i^l\mid x_j^l)_a<R$, we have $x_i\neq x_j$. 
 Since $(f(x_i^l)\mid f(x_j^l))_b\geq l\to \infty$, we have
 $\partial f(x_i)=\partial f(x_j)$.
 Therefore $\partial f$ is not $n$-to-one.

 Next, we suppose that (\ref{item:1}) is negative, and we will show that (\ref{item:2}) is negative.
 There exist $n+1$ points $a_1,\dots,a_{n+1}\in \partial X$, 
 which are mutually different, such that $\partial f(a_i) = \partial f(a_j)$
 for all $i,j$. Set $z:=\partial f(a_1)$. For $t> 0$, we define
 \begin{align*}
  U(a_i,t):=\{x\in X\cup \partial X: (x\mid a_i)_a\geq t\}.
 \end{align*}

 Let $\DA$ be a constant satisfying the statement of Lemma~\ref{lem:univ-const} for 
 both $X$ and $Y$.
 We take $R>0$ satisfying $U(a_i,\DA^{-1} R)\cap U(a_j,\DA^{-1} R)=\emptyset$ 
 for all $i\neq j$. We will see that for all $S>0$,
 we have $U(a_i,R)\cap f^{-1}(U(z,\DA S)) \neq \emptyset$. 

 Let $\gamma_i\in \xcalL{X}_a^\infty$ with $\gamma_i\in a_i$.
 By Corollary~\ref{cor:visu-bdl-cont}, we have 
 $z=\partial f(a_i) = \lim_{n\to \infty}f(\gamma_i(n))$.
 Since $(a_i\mid \gamma_i(n))_a\to \infty$, for sufficiently large $n$, 
 we can assume that 
 $\gamma_i(n)\in U(a_i,R)$. Similarly, by 
 $(z\mid f(\gamma_i(n)))_b\to \infty$, for sufficiently large $n$, 
 we can assume that $f(\gamma_i(n))\in U(z,\DA S)$. 
 Then $\gamma_i(n)\in U(a_i,R)\cap f^{-1}(U(z,\DA S)) \neq \emptyset$.

 For each $i$, we choose $x_i\in U(a_i,R)\cap f^{-1}(U(z,\DA S))$. 
 Then we have $(x_i\mid x_j)< R$ for all $i\neq j$.
 Indeed, if we suppose $(x_i\mid x_j)\geq R$, then
 \begin{align*}
  (a_i\mid x_j)_a\geq \DA^{-1}\min\{(a_i\mid x_i)_a, (x_i\mid x_j)_a\}
  \geq \DA^{-1}  R,
 \end{align*}
 so $x_j\in U(a_i,\DA^{-1} R)$, and this contradicts the assumption of $R$. 

 On the other hand, since $f(x_i)\in U(z,\DA S)$, we have
 \begin{align*}
  (f(x_i)\mid f(x_j))_b\geq \DA^{-1}\min\{(f(x_i)\mid z)_b, (z\mid f(x_j))_b\}\geq S.
 \end{align*}
 Therefore (\ref{item:2}) is negative.
\end{proof}
Miyata and Virk~\cite{Miyata-Virk-dim-rais-large}
introduced the following notion, which they called $(\mathrm{B})_n$. 
Later, Dydak and Virk~\cite{dydak-G-bdry-maps} reintroduced and 
they named it coarsely $n$-to-one.
\begin{definition}
 Let $X$ and $Y$ be metric spaces. We say that a coarse map 
 $f\colon X\to Y$ is \textit{coarsely $n$-to-one}  if
 for all $R>0$, there exists $S>0$ such that for all subset $A\subset Y$ with $\diam (A)<R$, there exists
 $B_1,\dots, B_n\subset X$ satisfying $\diam(B_i)<S$ for all $i$ and $f^{-1}(A)= B_1\cup \dots \cup B_n$.
\end{definition}

We use the following characterisation of coarsely $n$-to-one maps due to Dydak and Virk.
\begin{proposition}[{\cite[Proposition 7.3.]{dydak-G-bdry-maps}}]
\label{prop:cn-2-one-RS}
 Let $X$ and $Y$ be metric spaces and let $f\colon X\to Y$ be a bornologous map. Then the following is equivalent.
 \begin{enumerate}[$($A$)$]
  \item $f$ is coarsely $n$-to-one.
  \item \label{item:cn2one-RS}
	For all $R>0$, there exists $S>0$ such that, for all $n+1$ points $x_1,\dots, x_{n+1}\in X$ with
	$\ds{x_i,x_j}> S$ for all $i\neq j$,
	there exists $1\leq k<m\leq n+1$ such that $\ds{f(x_k),f(x_m)}>R$.
 \end{enumerate}
\end{proposition}

\begin{theorem}
 \label{thm:coarseN-to-one}
 Let $X$ and $Y$ be proper coarsely convex spaces and let $f\colon X\to Y$ be a visual, 
 $\rho$-$T$-$\xcalL{X}_a^\infty$-$\xcalL{Y}_b^\infty$-equivariant map. If $f$ is coarsely $n$-to-one, then the induced map 
 $\partial f\colon \partial X\to \partial Y$ is $n$-to-one.
\end{theorem}

\begin{proof}
 We use same setting for $X$ and $Y$ as those stated in the beginning of
 Section~\ref{sec:visual-map-radial}. 

 We suppose that $\partial f$ is not $n$-to-one. Thus there exist mutually different $n+1$ elements 
 $a_1,\dots,a_{n+1}\in \partial X$ such that $\partial f(a_1)=\partial f(a_2) = \dots =\partial f(a_{n+1})$.

 We choose $\gamma_i\in \xcalL{X}_a^\infty$ satisfying $\gamma_i\in a_i$ 
 for each $i=1,\dots,n+1$. Since $f$ is 
 $\rho$-$T$-$\xcalL{X}_a^\infty$-$\xcalL{Y}_b^\infty$-equivariant, 
 for each $\gamma_i$, there exists $\eta_i\in \xcalL{Y}_{b}^\infty$
 such that for all $t\in \Rp$, 
 \begin{align*}
  \ds{f\circ \gamma_i(t),\eta_i(\rho(t))}< T.
 \end{align*}
 Now set $R:=D'+2T$ and let $S>0$ be a constant satisfying (\ref{item:cn2one-RS}) in Proposition~\ref{prop:cn-2-one-RS}
 for the coarsely $n$-to-one map $f$. 
 Since $a_1,\dots,a_{n+1}$ are mutually different, we can take $s\in \Rp$ such that $\ds{\gamma_i(s),\gamma_j(s)}>S$ for all $i\neq j$.
 Then by Proposition~\ref{prop:cn-2-one-RS}, there exists a pair of indices $k,m$ such that 
 \begin{align*}
  \ds{f\circ \gamma_k(s),f\circ \gamma_m(s)}>R=D'+2T.
 \end{align*}
 Therefore 
 \begin{align*}
  \ds{\eta_k(\rho(s)),\eta_m(\rho(s))}&\geq \ds{f\circ \gamma_k(s),f\circ \gamma_m(s)}-\ds{f\circ \gamma_k(s),\eta_k(\rho(s))}
  -\ds{\eta_m(\rho(s)),f\circ \gamma_m(s)}\\
  &>D'+2T - 2T = D'.
 \end{align*}
 So by Lemma~\ref{lem:HdistD}, we have $[\eta_k]\neq [\eta_m]$. However, by Lemma~\ref{lem:LLoo-f}, we have 
 $[\eta_k]=\partial f(a_k) = \partial f(a_m) = [\eta_m]$. This is a contradiction.
\end{proof}

\begin{corollary}
\label{cor:dim-rise}
 Let $X$ and $Y$ be proper coarsely convex spaces. Suppose that there exists a visual 
 $\sigma$-$T$-$\xcalL{X}_a^\infty$-$\xcalL{Y}_b^\infty$-equivariant map $f\colon X\to Y$ 
 which is coarsely surjective and coarsely $(n+1)$-to-one.  Then
 \begin{align*}
  \dim \partial Y \leq \dim \partial X + n.
 \end{align*}
\end{corollary}
\begin{proof}
 By Proposition~\ref{prop:csurj-surj} and Theorem~\ref{thm:coarseN-to-one}, the induced map 
 $\partial f\colon \partial X\to \partial Y$ is surjective and $(n+1)$-to-one. Since both $\partial X$ and $\partial Y$
 are compact Hausdorff spaces, $\partial f$ is a closed map. Then 
 we obtain the desired inequality by dimension raising theorem~\cite[Theorem 4.3.3]{Eng-dim-theory}. 
\end{proof}

\begin{proof}[Proof of Theorem~\ref{thm:dim-rise}]
 By Theorem~\ref{thm:visual-Lip}, $f$ is visual.
 By Lemma~\ref{lem:LL-LLoo}, $f$ is 
 $\sigma$-$T$-$\xcalL{X}_a^\infty$-$\xcalL{Y}_b^\infty$-equivariant.
 Then Theorem~\ref{thm:dim-rise} follows from Corollary~\ref{cor:dim-rise}.
\end{proof}

\begin{corollary}
 Let $X$ and $Y$ be proper coarsely convex spaces, and let $f\colon X\to
 Y$ be a visual 
 $\sigma$-$T$-$\xcalL{X}_a^\infty$-$\xcalL{Y}_b^\infty$-equivariant 
 coarse embedding map.  Then $\partial f$ is a
 topological embedding. Moreover, if $f$ is coarsely surjective, then 
 $\partial f$ is a homeomorphism.
\end{corollary}

\begin{proof}
 Since $f$ is coarsely one-to-one, by Theorem~\ref{thm:coarseN-to-one},
 $\partial f$ is one-to-one. Since $\partial X$ and $\partial Y$ are
 compact Hausdorff spaces, $\partial f$ is a topological embedding. If $f$ is 
 coarsely surjective, by Proposition~\ref{prop:csurj-surj}, $\partial f$ 
 is surjective, thus $\partial f$ is a homeomorphism.
\end{proof}

\section{Example}
\label{sec:example}
Dydak and Virk constructed a example of coarsely 2-to-1 map from the Cayley
graph of the free group of rank 2 with respect to standard generating
set to the hyperbolic plane. In this section, by taking direct products
of their examples and some coarsely convex spaces, we construct examples
of coarsely 2-to-1 map between coarsely convex spaces satisfying the
assumptions in Theorem~\ref{thm:coarseN-to-one}. First we prepare some
lemmas.

\subsection{Lemmas for direct products of coarsely convex spaces}
\begin{lemma}
\label{lem:prod-metricsp}
 For $i=1,2$, let $X_i$ and $Y_i$ be metric spaces and $f_i\colon X_i\to Y_i$ be maps. The following holds.
 \begin{enumerate}[$($i$)$]
  \item \label{item:3}
	If $f_1$ and $f_2$ are coarsely surjective, so is $f_1\times f_2$.
  \item \label{item:4}
	If $f_1$ and $f_2$ are large scale Lipschitz maps, so is $f_1\times f_2$.
  \item \label{item:5}
	If $f_1$ is coarsely $n$-to-one and $f_2$ is coarsely $m$-to-one, then
	$f_1\times f_2$ is coarsely $nm$-to-one.
 \end{enumerate}
 Moreover, if $X_i$ and $Y_i$ are coarsely convex and $\xcalL{X_i}$ are systems of good quasi-geodesics of $X_i$,
 the following holds.
 \begin{enumerate}[$($i$)$]
  \setcounter{enumi}{3} 
   \item \label{item:6}
	If $f_1$ is $\sigma_1$-$\xcalL{X_i}_{a_1}$-radial 
	 and $f_2$ is $\sigma_2$-$\xcalL{X_2}_{a_2}$-radial, 
	 where $\sigma_1,\sigma_2$ are affine maps,
	 then
	$f_1\times f_2$ is 
	 $\min\{\sigma_1,\sigma_2\}$-$\xcalL{X_1\times X_2}_{(a_1,a_2)}$-radial.
 \end{enumerate}

\end{lemma}

\begin{proof}
 (\ref{item:3}) and ~(\ref{item:4}) are clear and we leave the proof to the reader. 
 (\ref{item:5}) is due to Miyata and Virk~\cite[Proposition 3.4]{Miyata-Virk-dim-rais-large}.
 We suppose that $X_i$ and $Y_i$ are coarsely convex. Let $\xcalL{X_i}$ be 
 the systems of good quasi-geodesics of $X_i$. Then by \cite[Proposition 3.3]{FO-CCH}, all element in
 the system of good quasi-geodesics $\xcalL{X_1\times X_2}$ of $X_1\times X_2$ is of the following form:
 \begin{align*}
  \alpha_1\gamma_1\oplus\alpha_2\gamma_2:[0,a_1+a_2]\ni t\mapsto
  \left(\gamma_1\left(\alpha_1t\right),
  \gamma_2\left(\alpha_2t\right)\right) \in X_1\times X_2
 \end{align*}
 where $\gamma_i\in \xcalL{X_i}$, $\Dom \gamma_i=[0,a_i]$ and $\alpha_i=a_i/(a_1+a_2)$. Now it is easy to see that
 (\ref{item:6}) holds.
\end{proof}

For the topological structure of the ideal boundaries of products of
coarsely convex spaces, the following is known.

\begin{proposition}[{\cite[Proposition 4.26]{FO-CCH}}]
 Let $X$ and $Y$ be proper coarsely convex spaces. Then the ideal boundary $\partial (X\times Y)$
 of the product $X\times Y$ is homeomorphic to the join $\partial X \star \partial Y$.
\end{proposition}

Since the ideal boundary of the hyperbolic plane $\hyp$ is the circle $S^1$ and that of the real line $\R$ 
is the two point space $S^0$, the ideal boundary of the product $\hyp \times \R$ is the suspension 
$\Sigma S^1$ which is homeomorphic to the 2-dimensional sphere $S^2$. Similarly, let $\cay(F_2)$ be 
the Cayley graph of the free group of rank 2 with the standard generating set. Then the ideal 
boundary of $\cay(F_2)\times \R$ is the suspension of the Cantor set.

\subsection{Lemma for product of geodesic coarsely convex spaces}
For technical reason, in the rest of this section we only consider geodesic coarsely convex spaces.
For $i=1,2$, Let $X_i$ be geodesic $(C_i,\xcalL{X_i})$-coarsely convex spaces and 
$Y_i$ be geodesic $(C'_i,\xcalL{Y_i})$-coarsely convex spaces. Let $a_i\in X_i$ and $b_i\in Y_i$ be base points.

\begin{lemma}
\label{lem:prod-equiv}
 Let $f_i\colon X_i\to Y_i$ be maps preserving base points and distances from base points, that is, $f(a_i) = b_i$ and
 $\ds{b_i,f(x)} = \ds{a_i,x}$ for all $x\in X_i$. If $f_i$ is 
 $\id_{\Rp}$-$H$-$\xcalL{X_i}_{a_i}$-$\xcalL{Y_i}_{b_i}$-equivariant, then
 $f_1\times f_2$ is 
 $\id_{\Rp}$-$2H$-$\xcalL{X_1\times X_2}_{(a_1,a_2)}$-$\xcalL{Y_1\times Y_2}_{(b_1,b_2)}$-equivariant.
\end{lemma}

\begin{proof}
 Let $\gamma_1\oplus \gamma_2 \in \xcalL{X_1\times X_2}_{(a_1,a_2)}$. 
 Set $[0,c_i]=\Dom \gamma_i$. 
 Then $\Dom \gamma_1\oplus \gamma_2 = [0,c_1+c_2]$.
 We choose $\eta_i\in \xcalL{Y_i}_{b_i}$ such that $\Dom \eta_i = [0,d_i]$
 and $f_i(\gamma_i(c_i))= \eta_i(d_i)$. Then we have
 $\eta_1\oplus\eta_2\in \xcalL{Y_1\times Y_2}_{(b_1,b_2)}$ and 
 $f_1\times f_2(\gamma_1\oplus\gamma_2(c_1+c_2))=\eta_1\oplus \eta_2(d_1+d_2)$.
 Since $\gamma_i$ and $\eta_i$ are geodesics, and $f_i$ preserve 
 the distance from base points,
 \begin{align*}
  c_i = \ds{\gamma_i(c_i),a_i} = \ds{f_i(\gamma_i(c_i)),b_i}
  = \ds{\eta_i(d_i),b_i} = d_i.
 \end{align*}
 Moreover, since $f_i$ is $\id_{\Rp}$-$H$-$\xcalL{X_i}$-$\xcalL{Y_i}$-equivariant,
 \begin{align*}
  \ds{f_i(\gamma_i(s)),\eta_i(s)}< H \quad (\forall s\in [0,c_i]).
 \end{align*}
 Now set $e_i:=c_i/(c_1+c_2)$. For all $t\in [0,c_1+c_1]$, we have
 \begin{align*}
  \ds{f_1\times f_2(\gamma_1\oplus \gamma_2(t)),\eta_1\oplus\eta_2(t)}
  &=\ds{(f_1\circ \gamma_1(e_1t),f_2\circ \gamma_2(e_2t)),
  (\eta_1(e_1t),\eta_2(e_2t))}\\
  &=\ds{(f_1\circ \gamma_1(e_1t),\eta_1(e_1t))} +
  \ds{(f_2\circ \gamma_2(e_2t),\eta_2(e_2t))} \\
  &< 2H.
 \end{align*}
\end{proof}

\subsection{Example}
Dydak and Virk constructed a map $f\colon \cay(F_2)\to \hyp^2$ from the 
Cayley graph of the free group of rank 2 with respect to the 
standers generating set to the hyperbolic plane \cite{dydak-G-bdry-maps}. 
The map $f$ satisfies the following:
\begin{enumerate}[(i)]
 \item \label{item:7}
       $f$ is a large scale Lipschitz map.
 \item $f(e)=O$, here $e\in F_2$ is the unit and $O\in \hyp^2$ is the
       origin.
 \item \label{item:8}
       For all $x\in \cay(F_2)$, we have $\ds{x,e} = \ds{f(x),O}$.        
 \item \label{item:9}
       $f$ is coarsely two-to-one.
\end{enumerate}
By~(\ref{item:8}), $f$ is $\id_{\Rp}$-radial. 
%
Since both $\hyp^2$ and $\R$ is geodesic coarsely convex space,
by Proposition~\ref{prop:productspace} so is the product $\hyp^2\times \R$.
Then by Lemma~\ref{lem:prod-metricsp}, 
$f\times \id\colon \cay(F_2)\times \R\to \hyp^2\times \R$ is a 
large scale Lipschitz, coarsely two-to-one map. 
By Lemma~\ref{lem:prod-equiv} and Theorem~\ref{thm:visual-Lip}, 
$f\times \id$ is visual, and by Theorem~\ref{thm:coarseN-to-one}, 
the induced map 
$\partial(f\times \id)\colon \partial(\cay(F_2)\times \R)\to S^2$
is two-to-one, where $\partial(\cay(F_2)\times \R)$ is homeomorphic to
the suspension of the Cantor set. Similarly, 
\begin{align*}
 f\times f\colon  \cay(F_2)\times \cay(F_2)\to \hyp^2\times \hyp^2
\end{align*}
is a large scale Lipschitz, coarsely four-to-one map, and induced map
\begin{align*}
 \partial(f\times f)\colon \partial(\cay(F_2)\times \cay(F_2))
 \to S^3
\end{align*}
is four-to-one map, where $\partial(\cay(F_2)\times \cay(F_2))$
is homeomorphic to the topological join of two Cantor sets.

\section{Radial extension}
\label{sec:radial-extension}
In this section, we will show that continuous maps between
the boundaries of coarsely convex spaces extend to visual maps
between those coarsely convex spaces. The idea is the following. Let
$X$ and $Y$ be coarsely convex spaces and let 
$f\colon \partial X\to \partial Y$ be a continuous map.
We first construct a map 
$\calO f\colon \calO \partial X\to \calO \partial Y$, where 
$\calO \partial X$ and $\calO \partial Y$ are open cones over $\partial X$,
and $\partial Y$, respectively. 
Then the extension might be obtained by a composition of
the ``logarithmic map'' 
$\log^\epsilon\colon X\to \calO \partial X$, $\calO f$, 
and the ``exponential map'' $\exp_\epsilon \colon \calO \partial Y\to Y$.
However, in general, the map $\log^\epsilon$ is defined only on the subset of
$X$, and the composite $\exp_\epsilon \circ \calO f$ is not necessarily 
bornologous. So we need some modifications.


\subsection{Open cone}
\label{sec:open-cone}
Let $Z$ be a compact metrizable space. 
The {\itshape open cone} over $Z$, denoted by $\calO Z$, is
the quotient $\Rp\times Z/(\{0\}\times Z)$. 
For $(t,z)\in \Rp\times Z$,
we denote by $tz$ the point in $\calO Z$ represented by $(t,z)$.

Let $d_Z$ be a metric on $Z$. 
We assume that the diameter of $Z$ is at most 2.
We define a metric $d_{\calO Z}$ on $\calO Z$ by 
\begin{align*}
 d_{\calO Z}(tx,sy):=\abs{t-s} + \min\{t,s\}d_Z(x,y).
\end{align*}
The function $d_{\calO Z}$ does satisfy the triangle inequality. 
See~\cite{WillettThesis} for details.
We call $d_{\calO Z}$ the induced metric by $d_Z$.

\begin{definition}
\label{def:calOf}
 Let $Z$ and $W$ be compact metric spaces and let $f\colon Z\to W$ 
 be a continuous map. We define the map 
 $\calO f\colon \calO Z \to \calO W$ by 
 $\calO f(tz):= t f(z)$.
\end{definition}

We remark that $\calO f$ is proper, but not necessarily bornologous.
However, by composing with the following radial contraction,
we will obtain a coarse map. 

\begin{definition}
 Let $r\colon \Rp \to \Rp$ be a rough contraction.
 The radial contraction associated
 to $r$ is the map $\radcon\colon \calO Z\to \calO Z$ defined by
 $\radcon(tz):= r(t)z$ for $tz\in \calO Z$.
\end{definition}

The following proposition says that for 
any proper continuous map from the open cone, 
by composing an appropriate 
radial contraction, we obtain a coarse map. 
In fact, being metrically proper and pseudocontinuous is enough 
to this result. We refer~\cite[Section 4]{MR1388312} and 
~\cite[Definition 5.2]{FO-CCH} for the definition of 
pseudocontinuous maps. We remark that continuous maps are pseudocontinuous.

\begin{proposition}[{\cite[Lemma 4.12]{MR1388312}}]
\label{prop:psedoconti}
 Let $V$ be a metric space, and let $F\colon \calO Z\to V$ 
 be a metrically proper pseudocontinuous map.
 Then there exists a 0-rough contraction $r\colon \Rp\to\Rp$ such that,
 for the radial contraction $\radcon \colon \calO Z \to \calO Z$
 associated to $r$, 
 the composite $F \circ \radcon\colon \calO Z \to V$ is a coarse map.
\end{proposition}

\begin{remark}
\label{rem:r-change}
 Let $\epsilon>0$ be a constant which will be fixed in the beginning of Section~\ref{sec:expon-logar-map}.
 We can assume without loss of generality that the function $r$ 
 in Proposition~\ref{prop:psedoconti} satisfies the following. 
 For $t,t'\in \Rp$ with $t^{\epsilon},t'^{\epsilon}\in r^{-1}([1,\infty))$, 
 we have
 \begin{align*}
  \abs{r(t^\epsilon)^{1/\epsilon}-r(t'^\epsilon)^{1/\epsilon}} \leq \abs{t-t'}.
 \end{align*}
 Indeed, we define the map $r'\colon \Rp\to\Rp$ by $r'(t):=r(t)$ for 
 $t\in r^{-1}([0,1))$, and $r'(t):=r(t)^\epsilon$ for $t\in r^{-1}([1,\infty))$. 
 Then $r'(t)\leq r(t)$ for all $t\in \Rp$ and 
 for $t,t'\in \Rp$ with $t^{\epsilon},t'^{\epsilon}\in r^{-1}([1,\infty))$, 
 \begin{align*}
  \abs{r'(t^\epsilon)^{1/\epsilon}-r'(t'^\epsilon)^{1/\epsilon}} 
  =\abs{r(t^\epsilon)-r(t'^\epsilon)} 
  \leq \abs{t^\epsilon-t'^\epsilon}\leq \abs{t-t'}.
 \end{align*}
 If the radial contraction associated to $r$ satisfies the statement
 of Proposition~\ref{prop:psedoconti}, so does the one associated to $r'$.
 Therefore we can replace $r$ by $r'$.
\end{remark}

\begin{lemma}
\label{lem:r-eps-rough-contr}
 Let $r\colon \Rp\to\Rp$ be a 0-rough contraction satisfying 
 Remark~\ref{rem:r-change}. 
 Set $\delta:=\sup\{u\in \Rp:r(u^\epsilon)\leq 1\}$.
 For $u,u'\in \Rp$, we have 
 \begin{align*}
  \abs{r(u^\epsilon)^{1/\epsilon} -r(u'^\epsilon)^{1/\epsilon}}\leq 
  \abs{u-u'}+2\delta. 
  \end{align*}
 Especially, the map $t\mapsto r(t^\epsilon)^{1/\epsilon}$ 
 is a $2\delta$-rough contraction.
\end{lemma}
\begin{proof}
 If $\min\{r(u^\epsilon),r(u'^\epsilon)\}\leq 1$, then
 \begin{align*}
  \abs{r(u^\epsilon)^{1/\epsilon} - r(u'^\epsilon)^{1/\epsilon}}
   \leq r(u^\epsilon)^{1/\epsilon} + r(u'^\epsilon)^{1/\epsilon}
  \leq u+u'
   \leq \abs{u-u'} + 2\delta.
 \end{align*}
 If $\min\{r(u^\epsilon),r(u'^\epsilon)\}> 1$, 
 then by Remark~\ref{rem:r-change},
 \begin{align*}
  \abs{r(u^\epsilon)^{1/\epsilon} - r(u'^\epsilon)^{1/\epsilon}}
  \leq \abs{u-u'}.
 \end{align*}
\end{proof}


\subsection{Exponential and logarithmic map}
\label{sec:expon-logar-map}
In the rest of this section, let $X$ and $Y$ be proper 
$(\lambda,k,E,C,\theta,\calL)$-coarsely convex spaces.
We fix a positive number $0 < \epsilon < 1$ satisfying the statement
of Proposition~\ref{prop:q-met2met}.
Let $d_{\epsilon,\partial_a X}$ and $d_{\epsilon,\partial_b Y}$ 
be a metric given by Proposition~\ref{prop:q-met2met} 
for $\partial_a X$ and $\partial_b Y$, respectively.
Let $K$ be the constant in the statement of Proposition~\ref{prop:q-met2met}.

First, we review the construction of the exponential map and the logarithmic map
for coarsely convex spaces. We refer~\cite[Section 5]{FO-CCH} for details.

The exponential map $\expe \colon \OdbY \to Y$ is defined as follows. 
For each $y\in \partial_b Y$, we choose a quasi-geodesic ray
$\eta_y\in \xcalL{Y}_b^\infty$ with $\eta_y\in y$. 
Then for $t\in \Rp$, we define 
$\expe(ty):=\eta_y(t^{\frac{1}{\epsilon}})$. 

We remark that $\expe$ is proper,
however, not necessarily bornologous. It is shown 
that $\expe$ is pseudocontinuous (\cite[Lemma 5.4]{FO-CCH}).




Now we introduce the logarithmic map. 
We define the subspace $\Xvis\subset X$ as follows.
\begin{align*}
 \Xvis:=\bigcup_{\gamma\in \xcalL{X}_a^\infty} \gamma(\Rp).
\end{align*}

The logarithmic map $\loge \colon \Xvis \to \OdaX$ 
is defined as follows. For $v\in \Xvis$, we choose a geodesic ray
$\gamma_v\in \xcalL{X}_a^\infty$ and a parameter $t_v\in \Rp$ such that 
$\gamma_v(t_v) = v$. Then we define $\loge(v):=t^{\epsilon}[\gamma_v]$.
The map $\loge$ is a coarse map (\cite[Proposition 5.6]{FO-CCH}).

Since the domain of $\loge$ is a subspace of $X$,
we need to deform $X$ to the subspace $\Xvis$.

\begin{proposition}[{\cite[Section 5.5]{FO-CCH}}]
\label{prop:deform}
 There exists a 1-rough contraction $\chi\colon \Rp\to\Rp$ such that
 the map $\varphi\colon X\to X$
 given by
 \begin{align*}
  \varphi(v):= \gamma_v(\chi(T_v)) \quad (v\in X).
 \end{align*}
 is a coarse map, and $\varphi(X)\subset \Xvis$.
 Here $v\in X$, $\gamma_v\in \xcalL{X}_a$, $T_v\in \Rp$, 
 and $v=\gamma_v(T_v)$. Moreover, $\varphi$ is coarsely homotopic 
 to the identity.
\end{proposition}

\subsection{Radial extension}
\label{sec:radial-extension-1}
In the rest of this section, 
let $f\colon \partial_a X\to \partial_b Y$ be a continuous map.


Let $\calO f\colon \OdaX \to \OdbY$ be the map given by 
Definition~\ref{def:calOf}. Since $\calO f$ is continuous and 
$\expe\colon \OdbY \to Y$ is pseudocontinuous, the composite 
$\expe\circ \calO f$ is pseudocontinuous. Thus by Proposition~\ref{prop:psedoconti},
there exists a 0-rough contraction 
$r\colon \Rp\to\Rp$ such that
the composite $\expe\circ \calO f \circ \radcon$ is a coarse map, 
where $\radcon$ is a radial contraction $\radcon\colon \OdaX\to \OdaX$ 
associated to $r$.

Now we define the map 
$\rad f \colon X\to Y$ by the composite 
\begin{align}
 \rad f:= \expe \circ \calO f \circ \radcon \circ \loge \circ \varphi
\end{align}
where $\varphi$ is the map given in Proposition~\ref{prop:deform}.
We call it 
the {\itshape radial extension of $f$ associated to the map $r$}.
Since $\rad f$ is the composite of coarse maps, it is a coarse map.
In Proposition~\ref{prop:chmtpr1r2}, we will show that the coarse homotopy
class of $\rad f$ does not depend on the choice of the map $r$.

The main purpose of this section is to prove the following.

\begin{theorem}
\label{thm:rad-ext} 
Let $X$ and $Y$ be coarsely convex space with system of good quasi-geodesics
 $\xcalL{X}$ and $\xcalL{Y}$, respectively. 
 Let $f\colon \partial X\rightarrow \partial Y$ be a continuous map. Then
 there exists a 0-rough contraction $r\colon \Rp\to \Rp$ 
  such that the radial extension
 $\rad f$ of $f$ associated to $r$ is $\xcalL{X}$-radial and 
 $\xcalL{X}$-$\xcalL{Y}$-equivariant. Moreover,
 \begin{enumerate}[$($i$)$]
  \item \label{item:10}
	For a continuous map $f\colon \partial X\rightarrow \partial Y$,
	the induced map $\partial \rad f\colon \partial X\to \partial Y$ is equal to $f$.
  \item \label{item:11}
 Let $F\colon X\to Y$ be a visual
 $\rho$-$T$-$\xcalL{X}_a^\infty$-$\xcalL{Y}_b^\infty$-equivariant map.
 Then $\rad \partial F$ is coarsely homotopic to $F$.
 \end{enumerate} 
\end{theorem}

In the rest of this section, let $\DA$ be a constant satisfying 
the statement of Lemma~\ref{lem:univ-const} for both $X$ and $Y$.

\begin{lemma}
 \label{lem:rate-conti}
 There exists a non-decreasing
 map $\psi\colon \Rp\to \Rp\cup\{\infty\}$ 
 with $\lim_{t\to\infty}\psi(t)= \infty$ 
 such that
 for $p,q \in \partial_a X$ with $p\neq q$, we have 
 \begin{align*}
  (f(p)\mid f(q))_b\geq \psi((p\mid q)_a).
 \end{align*}
\end{lemma}

\begin{proof}
 We define a map $\psi\colon \Rp\to\Rp\cup\{\infty\}$ by 
 \begin{align*}
  \psi(t):= \inf\{(f(p)\mid f(q))_b: 
	    (p,q)\in (\partial_a X)^2,\, (p\mid q)_a\geq t\},
  \quad (t\geq 0).
 \end{align*}
 By the definition, $\psi$ is non-decreasing. 
 Since $\partial_a X$ is compact, $f$ is uniformly continuous. 
 So by Proposition~\ref{prop:q-met2met}, 
 we have $\lim_{t\to \infty}\psi(t) = \infty$.
\end{proof}

We remark that if $\partial_a X$ has at least one accumulation point,
then the range of the map $\psi$ in Lemma~\ref{lem:rate-conti}
does not contain $\infty$.

Set 
$\radp f:=\expe \circ \calO f \circ \radcon \circ \loge \colon \Xvis\to Y$.
First we consider this map.



\begin{proposition}
\label{prop:radp-LLequiv}
 Suppose that the map $r\colon \Rp\to \Rp$ satisfies
 \begin{align*}
  r\left(t^\epsilon\right)^{1/\epsilon}\leq 
  \DA^{-1}
  \psi\left(
   \frac{\DA^{-1}(t-\bar{\theta}(0))}{E(\lambda\bar{\theta}(0) + k_1)}
  \right) \quad (\forall t\in \Rp).
 \end{align*}
 Then the map $\radp f$ is $\xcalL{X}$-radial and $\xcalL{X}$-$\xcalL{Y}$-equivariant.
\end{proposition}

We remark that $\Xvis$ is a subspace of the coarsely convex space,
however, both being $\sigma$-$\xcalL{X}_a$-radial and being
$\rho$-$H$-$\xcalL{X}_a$-$\xcalL{Y}_b$-equivariant make sense for the map
$\radp f$.

\begin{proof}
 Let $x\in \Xvis$. We choose $\gamma_x\in \xcalL{X}_a^\infty$, $T_x\in \Rp$
 such that $x=\gamma_x(T_x)$ and $\loge(x)=T_x^\epsilon[\gamma_x]$.
 For $[\gamma_x]$, we choose $\eta_x\in \xcalL{Y}_b^\infty$ such that
 $f([\gamma_x])=[\eta_x]$ and 
 $\expe(s[\eta_x])=\eta_x(s^{1/\epsilon})$ for $s\in \Rp$.
 Then we can compute $\radp f(x)$ as follows.
 \begin{align*}
  x\xmapsto{\loge} T_x^\epsilon[\gamma_x]
  \xmapsto{\radcon}
   r(T_x^{\epsilon}) [\gamma_x]
  \xmapsto{\calO f}
   r(T_x^{\epsilon})f([\gamma_x]) 
  \xmapsto{\expe}
   \eta_x(r(T_x^{\epsilon})^{1/\epsilon}).
 \end{align*}
 We have
 \begin{align}
  \label{eq:4}
  \ds{\eta_x(r(T_x^{\epsilon})^{1/\epsilon}),b}\geq 
  \frac{1}{\lambda}r(T_x^{\epsilon})^{1/\epsilon} - k_1.
 \end{align}
 Set 
 \begin{align*}
  \rho(t):=r(t^{\epsilon})^{1/\epsilon} \quad\text{and}\quad
  \sigma(t):=\left(\frac{1}{\lambda}\rho(t) - k_1\right)\vee 0.
 \end{align*}
 By (\ref{eq:4}), $\radp f$ is weakly $\sigma$-$\xcalL{X}_a$-radial, so 
 it is $\xcalL{X}$-radial, by Lemma~\ref{lem:weak-rad-rad}.

 Now let $\xi\in \xcalL{X}_a$ with $\Dom \xi = [0,L_{\xi}]$.
 We put $x:=\xi(L_{\xi})$. We use the above computation for $x$, 
 that is, $\radp f(x)= \eta_x(r(T_x^{\epsilon})^{1/\epsilon})$.
 Let $\eta\in \xcalL{Y}_b$ with $\Dom \eta =[0,L_\eta]$ such that 
 $\eta(L_\eta) = \eta_x(\rho(T_x))=\radp f(\xi(L_\xi))$.
 Since $\Dom \eta_x=\Rp$, 
 by Lemma~\ref{lem:univ-const}~(\ref{lem:ray-same-param}), 
 \begin{align}
 \label{eq:6}
  \ds{\eta_x(s),\eta(s)}\leq \DA, \quad (0\leq \forall s\leq L_\eta).
 \end{align} 

 We fix $t\in \Rp$ with $t\leq \min\{L_\xi, \rho^{-1}(L_\eta)\}$ 
 and set $w:=\xi(t)$.
 We choose $\gamma_w\in \xcalL{X}_a^\infty$ and $T_w\in \Rp$ such that
 $w=\gamma_w(T_w)$ and $\loge(w)=T_w^\epsilon [\gamma_w]$.
 Since $\gamma_w(T_w) = \xi(t)$, by Lemma~\ref{lem:prod-est-bigon},
 \begin{align*}
  (\gamma_w\mid \xi)_a
  \geq \frac{t-\tilde{\theta}(0)}{E(\lambda\tilde{\theta}(0) + k_1)}.
 \end{align*}
 Applying the same argument for $x=\gamma_x(T_x) = \xi(L_\xi)$, we have
 \begin{align*}
  (\gamma_x\mid \xi)_a 
  \geq \frac{L_\xi-\tilde{\theta}(0)}{E(\lambda\tilde{\theta}(0) + k_1)}.
 \end{align*}
 So we have
 \begin{align*}
  (\gamma_x\mid \gamma_w)_a\geq 
  \DA^{-1}\min\{(\gamma_x\mid \xi)_a, (\gamma_w\mid \xi)_a\}\geq 
  \frac{\DA^{-1}(t-\tilde{\theta}(0))}{E(\lambda\tilde{\theta}(0) + k_1)}.
 \end{align*}
 We choose $\eta_w\in \xcalL{Y}_b^\infty$ such that 
 $f([\gamma_w])=[\eta_w]$ and
 $\expe(s[\eta_w])=\eta_w(s^{1/\epsilon})$ for all $s\in \Rp$.
 By Lemma~\ref{lem:rate-conti},
 \begin{align*}
  ([\eta_x]\mid [\eta_w])_b&\geq \psi(([\gamma_x]\mid [\gamma_w])_a)
  \geq \psi\left(
   \frac{\DA^{-1}(t-\tilde{\theta}(0))}{E(\lambda\tilde{\theta}(0) + k_1)}
  \right).
 \end{align*}

 Since $\gamma_w(T_w)=\xi(t)$, we have $\abs{T_w-t}\leq \tilde{\theta}(0)$.
 By the assumption on $r$, we have
 \begin{align*}
  \rho(t) \leq 
  \DA^{-1}
  \psi\left(
   \frac{\DA^{-1}(t-\tilde{\theta}(0))}{E(\lambda\tilde{\theta}(0) + k_1)}
  \right)
  \leq (\eta_x\mid \eta_w)_b.
 \end{align*}
 Then by Lemma~\ref{lem:univ-const}~(\ref{lem:maximizer}), we have 
 \begin{align*}
  \ds{\eta_w(\rho(t)),\eta_x(\rho(t))}\leq \DA.
 \end{align*}
 Set $\delta:=\sup\{u\in \Rp:r(u^\epsilon)\leq 1\}$. 
 By Lemma~\ref{lem:r-eps-rough-contr}, 
 \begin{align*}
  \abs{\rho(T_w)-\rho(t)}\leq \abs{T_w-t}+2\delta
  \leq \tilde{\theta(0)} + 2\delta.
 \end{align*}
 Therefore 
 \begin{align}
  \label{eq:5}
  \ds{\radp f(\xi(t)),\eta_x(\rho(t))} &= \ds{\eta_w(\rho(T_w)),\eta_x(\rho(t))}
  \leq \ds{\eta_w(\rho(t)),\eta_x(\rho(t))} 
  + \lambda(\tilde{\theta}(0)+2\delta) + k_1\\
\notag
  &\leq \DA +\lambda(\tilde{\theta}(0)+2\delta) + k_1.
 \end{align}
 Set $H:=\DA + \lambda(\tilde{\theta}(0)+2\delta) + k_1+\DA$. 
%
%
 Since $\rho(t)\leq L_\eta$, by~(\ref{eq:5}) and (\ref{eq:6}), 
 we have
 \begin{align*}
  \ds{\radp f(\xi(t)),\eta(\rho(t))}\leq H.
 \end{align*}
 Hence the proof is done.                                                                                                                                          
\end{proof}

\begin{proposition}
\label{prop:vaphi-LLequiv}
 The map $\varphi \colon X\to X$
 is $\xcalL{X}$-radial and  $\xcalL{X}$-$\xcalL{Y}$-equivariant.
\end{proposition}

\begin{proof}
 Let $x\in X$.  We choose $\gamma_x\in \xcalL{X}$ and $T_x\in \Rp$ such that
 $x=\gamma_x(T_x)$, $\varphi(x)=\gamma_x(\chi(T_x))$.
 Then
 \begin{align*}
  \ds{\varphi(x),a}=\ds{\gamma_x(\chi(T_x)),\gamma_x(0)} 
  \geq \frac{1}{\lambda}\chi(T_x) -k.
 \end{align*}
 Set $\sigma(t):=(\lambda^{-1}\chi(t)-k)\vee 0$. Then
 $\varphi$ is weakly $\sigma$-$\xcalL{X}_a$-radial, so 
 it is $\xcalL{X}$-radial, by Lemma~\ref{lem:weak-rad-rad}.

 Let $\gamma\in \xcalL{X}$. Set $\Dom \gamma = [0,L_\gamma]$ and 
 $x:=\gamma(L_\gamma)$. 
 Let $\eta\in \xcalL{X}_a$ with $\Dom \eta = [0,L_\eta]$ such that 
 $\eta(L_\eta) = \varphi(x)$.

 We fix $t\in \Rp$ with $t\leq \min\{L_\gamma,\chi^{-1}(L_\eta)\}$,
 and set $w:=\gamma(t)$.
 We choose $\gamma_x, \gamma_w\in \xcalL{X}_a$ and $T_x,T_w\in \Rp$ such that
 $x=\gamma_x(T_x)$, $w=\gamma_w(T_w)$, $\varphi(x)=\gamma_x(\chi(T_x))$, 
 and $\varphi(w)=\gamma_w(\chi(T_w))$. 

 Since $\gamma(t)=\gamma_w(T_w)$,
 we have $\abs{T_w-t}\leq \theta(0)$.
 Set $u:=\min\{T_w,t\}$. 
 We remark
 \begin{align*}
  \abs{\chi(T_w)-\chi(u)}
  &\leq \abs{T_w-u}+1 \leq \theta(0) + 1,\\
  \abs{\chi(t)-\chi(u)}
  &\leq \abs{t-u}+1 \leq \theta(0) + 1.
 \end{align*}
 By Lemma~\ref{lem:univ-const}~(\ref{lem:ray-same-param}),
 \begin{align*}
  \ds{\gamma_w(\chi(u)),\gamma(\chi(u))}
  \leq \DA.
 \end{align*}
 Therefore we have
 \begin{align}
  \label{eq:13}
  \ds{\gamma_w(\chi(T_w)),\gamma(\chi(t))} 
  &\leq \ds{\gamma_w(\chi(u)),\gamma(\chi(u))} + \lambda(\theta(0) +1) +k\\
\notag
  &\leq \DA + \lambda\theta(0)+\lambda+ k.
 \end{align}

 Now, first we suppose $\chi(t)\leq T_x$. Then,
 since $\gamma(L_\gamma) =x= \gamma_x(T_x)$,
 by Lemma~\ref{lem:univ-const}~(\ref{lem:ray-same-param}),
 \begin{align*}
  \ds{\gamma(\chi(t)),\gamma_x(\chi(t))}
  \leq \DA.
 \end{align*}
 Since $\gamma_x(\chi(T_x)) = \eta(L_\eta)$ and 
 $\chi(t)\leq L_\eta$, by 
 Lemma~\ref{lem:univ-const}~(\ref{lem:ray-same-param}) we have 
 \begin{align*}
  \ds{\gamma_x(\chi(t)),\eta(\chi(t))}
  \leq \DA.
 \end{align*}
 Thus we have
 \begin{align}
  \label{eq:18}
  \ds{\gamma_w(\chi(T_w)),\eta(\chi(t))}\leq 3\DA  
  + \lambda\theta(0) +k+\lambda.
 \end{align}

 Next we suppose $\chi(t)> T_x$. Then, we have 
 \begin{align*}
  T_x< \chi(t)\leq t\leq L_\gamma \quad \text{and}
  \quad 
  \chi(T_x)\leq T_x< \chi(t)\leq L_\eta.
 \end{align*}
 Since  $\gamma_x(T_x)=\gamma(L_\gamma)$ and 
 $\gamma_x(\chi(T_x))=\eta(L_\eta)$, 
 \begin{align*}
  \abs{T_x-L_\gamma} \leq  \theta(0) \quad \text{and}
  \quad
  \abs{\chi(T_x)-L_\eta} \leq  \theta(0).
 \end{align*}
 Then we have
 \begin{align*}
  \ds{\gamma(\chi(t)),\eta(\chi(t))}&\leq 
  \ds{\gamma(\chi(t)),\gamma(L_\gamma)} 
  + \ds{\gamma_x(T_x),\gamma_x(\chi(T_x))} 
  + \ds{\eta(L_\eta),\eta(\chi(t))}\\
  &\leq \lambda(\abs{\chi(t)-L_\gamma} + \abs{T_x-\chi(T_x)} 
  + \abs{L_\eta-\chi(t)}) + 3k\\
  &\leq \lambda(\abs{T_x-L_\gamma}+2\abs{\chi(T_x)-L_\eta})+3k\\
  & \leq 3\lambda\theta(0)+3k.
 \end{align*}
 Thus by (\ref{eq:13}) and the above inequality, we have
 \begin{align}
  \label{eq:15}
  \ds{\gamma_w(\chi(T_w)),\eta(\chi(t))}\leq 
    \DA + 4(\lambda\theta(0) +k) + \lambda.
 \end{align}

 Finally by (\ref{eq:18}) and (\ref{eq:15}) we have
 \begin{align*}
  \ds{\varphi(\gamma(t)),\eta(\chi(t))}&=
  \ds{\varphi(w),\eta(\chi(t))} =  \ds{\gamma_w(\chi(T_w)),\eta(\chi(t))}\\
  &\leq 3\DA + 4(\lambda\theta(0) +k) +\lambda.
 \end{align*}
Hence the proof is done.
\end{proof}

The first half of Theorem~\ref{thm:rad-ext} follows from Proposition~\ref{prop:radp-LLequiv}
and Proposition~\ref{prop:vaphi-LLequiv}.

\begin{proposition}
\label{prop:chmtpr1r2}
 For $i=1,2$, let $r_i\colon \Rp\to \Rp$ be 0-rough contractions satisfying
 Remark~\ref{rem:r-change}.
 Let $\rad_{i} f$ be the radial extension of $f$ associated to $r_i$.
 Then $\rad_{1} f$ and $\rad_{2} f$ are coarsely homotopic.
\end{proposition}

\begin{proof}
 Let $\radconi_i$ be the radial contractions $\radconi_i\colon \OdaX\to \OdaX$ 
 associated to $r_i$ such that 
 $\rad_{i}=\expe\circ \calO f\circ \radconi_i \circ \loge\circ \varphi$.

 We can assume without loss of generality that
 \begin{align}
  \label{eq:17}
  r_1(t)\leq r_2(t) \quad (\forall t\in \Rp).
 \end{align}
 Indeed, set $r_3(t):=\min\{r_1(t),r_2(t)\}$ and let $\rad_3 f$ be the
 radial extension of $f$ associated to $r_3$.
 If we showed that $\rad_i f$ and $\rad_3 f$ are coarsely homotopic,
 then by the transitivity of coarse homotopy, $\rad_1 f$ and $\rad_2 f$
 are coarsely homotopic.

 Now set 
 $\radp_{i}:= \expe\circ \calO f\circ \radconi_i \circ \loge\colon \Xvis\to Y$.
 It is enough to show that $\radp_1 f$ and $\radp_{2} f$ 
 are coarsely homotopic.
 
 For each $x\in \Xvis$, we choose $\gamma_x\in \xcalL{X}_a^\infty$ and 
 $T_x\in\Rp$ 
 such that $\loge(x)=T_x^\epsilon [\gamma_x]$. 
 We also choose $\eta_x\in \xcalL{Y}_b^\infty$ such that 
 $f([\gamma_x]) = [\eta_x]$ and 
 $\expe(s[\eta_x])=\eta_x(s^{1/\epsilon})$ for all $s\in \Rp$.

 Set $Z:=\{(x,t)\in \Xvis \times\Rp\colon t\leq T_x\}$. We define a coarse homotopy
 $H\colon Z\to Y$ by
 \begin{align*}
  H(x,t)= \eta_x(r_1((T_x-t)^\epsilon)^{1/\epsilon} 
  + r_2(t^\epsilon)^{1/\epsilon}).
 \end{align*}
 Then for $x\in \Xvis$, $H(x,0) = \radp_1 f(x)$ and $H(x,T_x)=\radp_2 f(x)$.
 So all we need to show is that $H$ is a coarse map.
 It is easy to show that $H$ is metrically proper. We will show that
 $H$ is bornologous. 

 Let $(x,t),(y,t')\in Z$.
 For $y$, we choose $\gamma_y\in \xcalL{X}_a^\infty$, 
 $T_y\in\Rp$ and $\eta_y\in \xcalL{Y}_b^\infty$ as above.
 Set $\delta:=\sup\{u\in \Rp:r_1(u^\epsilon)\leq 1\}$.
 By Lemma~\ref{lem:r-eps-rough-contr}, for $i=1,2$, we have 
\begin{align}
 \label{eq:19}
  \abs{r_i(u^\epsilon)^{1/\epsilon} -r_i(u'^\epsilon)^{1/\epsilon}}\leq 
 \abs{u-u'}+2\delta, \quad (\forall u,u'\in \Rp).
\end{align}

 Now set $s:= r_1((T_x-t)^\epsilon)^{1/\epsilon}  + r_2(t^\epsilon)^{1/\epsilon}$ and 
$s':= r_1((T_y-t')^\epsilon)^{1/\epsilon}  + r_2(t'^\epsilon)^{1/\epsilon}$.
 By~(\ref{eq:19}), we have
 \begin{align*}
  \abs{s-s'}&\leq 
  \abs{r_1((T_x-t)^\epsilon)^{1/\epsilon} 
   - r_1((T_y-t')^\epsilon)^{1/\epsilon}}
   + \abs{r_2(t^\epsilon)^{1/\epsilon} - r_2(t'^\epsilon)^{1/\epsilon}}\\
  &\leq \abs{T_x - T_y} + 2\abs{t-t'} + 4\delta\\
  &\leq \tilde{\theta}(\ds{x,y}) + 2\abs{t-t'} + 4\delta.& 
 \end{align*}
 Set $q:=\min\{s,s'\}$. 
 Then by  Lemma~\ref{lem:t-ab} and (\ref{eq:17}), we have
 \begin{align*}
  \ds{H(x,t),H(y,t')}=&\ds{\eta_x(s),\eta_y(s')}\\
  =&\ds{\eta_x(q),\eta_y(q)} + \lambda\abs{s-s'}+k_1\\
  \leq &
  E'\left(\ds{\eta_x(r_2(T_x^\epsilon)^{1/\epsilon}),
  \eta_y(r_2(T_y^\epsilon)^{1/\epsilon})}
  +\lambda'\tilde{\theta}(\ds{\eta_x(r_2(T_x^\epsilon)^{1/\epsilon}),
  \eta_y(r_2(T_y^\epsilon)^{1/\epsilon}})\right) \\
& + D' + \lambda(\tilde{\theta}(\ds{x,y})  + 2\abs{t-t'} + 4\delta) +k'_1.\\
  =& E'\left(\ds{\radp_2 f(x),\radp_2 f(y)} 
  + \lambda'\tilde{\theta}(\ds{\radp_2 f(x),\radp_2 f(y)})\right) \\
  &+ D' + \lambda(\tilde{\theta}(\ds{x,y})  + 2\abs{t-t'} + 4\delta) +k'_1.
 \end{align*}
 Since $\radp_2 f$ is bornologous, $\ds{\radp_2 f(x),\radp_2 f(y)}$ is 
 bounded from the above by a constant depending only on $\ds{x,y}$.
 Therefore $H$ is bornologous. This complete the proof.
\end{proof}

\subsection{Induced map of the radial extension}
In this section, we give a proof of the statement (\ref{item:10})
of Theorem~\ref{thm:rad-ext}.
Let $f\colon \partial_a X\to \partial_b Y$ be 
a continuous map and let $\rad f$ be its radial extension.
We will show that the induced map $\partial \rad f$ is equal to $f$.

\begin{lemma}
 We have $\partial \radp f = f$.
\end{lemma}

\begin{proof}
 Let $\gamma\in \xcalL{X}_a^\infty$.  We choose $\eta\in \xcalL{Y}_b^\infty$
 such that $[\eta] = f([\gamma])$ and 
 $\eta(s)=\expe(s[\eta])$ for all $s\in \Rp$.
 
 Set $x_n:=\gamma(n)$ for $n\in \N$.
 We choose $\gamma_n\in \xcalL{X}_a^\infty$ and $T_n\in \Rp$ 
 such that $\gamma_n(T_n)=x_n = \gamma(n)$ and $\loge(x_n)=T_n[\gamma_n]$.
 By Lemma~\ref{lem:prod-est-bigon}, 
 \begin{align*}
  (\gamma_n\mid \gamma)_a \geq 
  \frac{n-\tilde{\theta}(0)}{E(\tilde{\theta}(0)+k_1)}.
 \end{align*}

 We choose $\eta_n\in \xcalL{Y}_b^\infty$ such that 
 $[\eta_n] = f([\gamma_n])$ and $\expe(s[\eta_n])=\eta_n(s)$ for $s\in \Rp$.
 Then $\radp f(x_n)=\eta_n(r(T_n^{\epsilon})^{1/\epsilon})$.

 By Lemma~\ref{lem:rate-conti}, 
 \begin{align*}
  ([\eta_n]\mid [\eta])_b&=(f([\gamma_n])\mid f([\gamma]))_b\\
  &\geq \psi(([\gamma_n]\mid [\gamma])_a)\\
  &\geq \psi\left(
  \frac{n-\tilde{\theta}(0)}{E(\tilde{\theta}(0)+k_1)}\right)
  \rightarrow \infty.
 \end{align*}
 Let $\eta'_n\in \xcalL{Y}_b$ with $\Dom \eta'_n=[0,L_{\eta'_n}]$ such that
 $\eta'_n(L_{\eta'_n})=\radp f(x_n)$. 
 By Lemma~\ref{lem:prod-est-bigon},
 \begin{align*}
  (\radp f(x_n)\mid [\eta_n])_b\geq (\eta'_n\mid \eta_n)_b
  \geq 
  \frac{r(T_n^{\epsilon})^{1/\epsilon}-\tilde{\theta}(0)}
       {E(\lambda \tilde{\theta}(0)+k_1)} \to \infty.
 \end{align*}
 So we have
 \begin{align*}
  (\radp f(x_n)\mid [\eta])_b
  \geq \DA^{-1}\min\{(\radp f(x_n)\mid [\eta_n])_b, ([\eta_n]\mid [\eta])_b\}
  \to \infty.
 \end{align*}
 Thus $\radp f(x_n)\to [\eta]$.
 Since $\radp f \cup \partial \radp f$ is continuous at $\partial_a X$,
 we have $\partial \radp f([\gamma]) = [\eta] =  f([\gamma])$.
\end{proof}

\begin{lemma}
 We have  $\partial \varphi = \id_{\partial_a X}$.
\end{lemma}

\begin{proof}
 Let $\gamma\in \xcalL{X}_a^\infty$. 
 Set $x_n:=\gamma(n)$ for $n\in \N$.
 We choose $\gamma_n\in \xcalL{X}_a^\infty$ and $T_n\in \Rp$ 
 such that $\gamma_n(T_n)=x_n = \gamma(n)$ and 
 $\varphi(x_n)=\gamma_n(\chi(T_n))$.
 Since $\gamma_n(T_n)= \gamma(n)$, by Lemma~\ref{lem:prod-est-bigon}, 
 \begin{align*}
 (\gamma_n\mid \gamma)_a \geq 
  \frac{n-\tilde{\theta}(0)}{E(\tilde{\theta}(0)+k_1)}\to \infty.
 \end{align*}
 Let $\eta_n\in \xcalL{X}_a$ with $\Dom \eta_n=[0,S_n]$ such that 
 $\eta_n(S_n) = \varphi(x_n)=\gamma_n(\chi(T_n))$.
 Then, by Lemma~\ref{lem:prod-est-bigon}, 
 \begin{align*}
  (\eta_n\mid \gamma_n)_a
  \geq \frac{\chi(T_n)-\tilde{\theta}(0)}{E(\tilde{\theta}(0)+k_1)}
  \to \infty.
 \end{align*}
 So we have
 \begin{align*}
  (\varphi(x_n)\mid [\gamma])_a\geq (\eta_n\mid \gamma)\geq 
  \DA^{-1}\min\{(\eta_n\mid \gamma_n)_a, (\gamma_n\mid \gamma)_a\}
  \to \infty.
 \end{align*}
 Thus $\varphi(x_n)\to [\gamma]$. 
 Since $\varphi\cup \partial \varphi$ is continuous at $\partial_a X$,
 we have $\partial \varphi([\gamma]) =  [\gamma]$.
\end{proof}

\begin{corollary}
\label{cor:pradf=f}
 $\partial \rad f = f$. Especially, $\rad f$ is visually $n$-to-one if and only if
 $f$ is $n$-to-one.
\end{corollary}
\begin{proof}
 The second half of the statement follows from Proposition~\ref{prop:bdry-n-to-one}.
\end{proof}


\subsection{Radial extension of the induced map}
\label{sec:raddelF}
In this section, we give a proof of the statement (\ref{item:11})
of Theorem~\ref{thm:rad-ext}.
Let $F\colon X\to Y$ be a visual
$\rho$-$T$-$\xcalL{X}_a^\infty$-$\xcalL{Y}_b^\infty$-equivariant map.
We will show the that the radial extension $\rad \partial F$ of the 
induced map $\partial F\colon \partial X\to \partial Y$ is coarsely homotopic 
to $F$.

 First, we will show that 
 $\radp \partial F\colon \Xvis\to Y$ is coarsely homotopic 
to the restriction $F|_{\Xvis}$.
 For each $x\in \Xvis$, we choose $\gamma_x\in \xcalL{X}_a^\infty$ and 
 $T_x\in\Rp$ 
 such that $\loge(x)=T_x^\epsilon [\gamma_x]$. 
 We also choose $\eta_x\in \xcalL{Y}_b^\infty$ such that 
 $\partial F([\gamma_x]) = [\eta_x]$ and 
 $\expe(s[\eta_x])=\eta_x(s^{1/\epsilon})$.

 Set $\Zvis:=\{(x,t)\in \Xvis \times\Rp\colon t\leq T_x\}$. 
 We define a coarse homotopy
 $H\colon \Zvis\to Y$ by
 \begin{align*}
  H(x,t)= \eta_x(\rho(T_x-t) + r(t^\epsilon)^{1/\epsilon}).
 \end{align*}

By Proposition~\ref{prop:chmtpr1r2}, we can assume without loss of
generality that $r(t^\epsilon)^{1/\epsilon}\leq \rho(t)$ for all 
$t\in \Rp$.

\begin{lemma}
\label{lem:H-coarse-map}
 The map $H$ is a coarse map.
\end{lemma}

\begin{proof}
  It is easy to show that $H$ is metrically proper, so we will show that $H$
 is bornologous.

 Let $(x,t),(y,t')\in \Zvis$. For $x$, $y$, 
 we choose $\gamma_x,\gamma_y\in \xcalL{X}_a^\infty$,
 $T_x,T_y\in \Rp$, $\eta_x,\eta_y\in \xcalL{Y}_b^\infty$ such that
 \begin{align*}
  x= \gamma_x(T_x),&\quad
  \loge(x)=T_x^\epsilon [\gamma_x], \quad \partial F([\gamma_x]) = [\eta_x], 
  \quad
 \expe(s[\eta_x])=\eta_x(s^{1/\epsilon}) \quad (\forall s),\\
  y= \gamma_(T_y),&\quad
  \loge(y)=T_y^\epsilon [\gamma_y], \quad  \partial F([\gamma_y]) = [\eta_y], 
  \quad
 \expe(s[\eta_y])=\eta_y(s^{1/\epsilon}) \quad (\forall s).
 \end{align*}
 
 Set $\delta:=\sup\{u\in \Rp:r(u^\epsilon)\leq 1\}$. 
 By Lemma~\ref{lem:r-eps-rough-contr}, we have
 \begin{align*}
  \abs{r(t^\epsilon)^{1/\epsilon} - r(t'^\epsilon)^{1/\epsilon}}
  \leq \abs{t-t'} + 2\delta.
 \end{align*}

 Let $\tau$ be a constant such that $\rho$ is a $\tau$-rough contraction.
 Then
 \begin{align*}
  \abs{\rho(T_x-t)-\rho(T_y-t')}
  \leq \abs{T_x-T_y}+\abs{t-t'}+\tau
  \leq \abs{t-t'}+ \tilde{\theta}(\ds{x,y}) + \tau.
 \end{align*}
 Set $s:=\rho(T_x-t)+r(t^\epsilon)^{1/\epsilon}$ and 
 $s':=\rho(T_y-t')+r(t'^\epsilon)^{1/\epsilon}$.
Then we have
 \begin{align*}
  \abs{s-s'} &\leq \abs{\rho(T_x-t) - \rho(T_y-t')}  
    + \abs{r(t^\epsilon)^{1/\epsilon} - r(t'^\epsilon)^{1/\epsilon}}\\
  &\leq \tilde{\theta}(\ds{x,y}) 
  + 2\abs{t-t'} + 2\delta + \tau.
 \end{align*}
 Now set $q:=\min\{s,s'\}$.
 Then by  Lemma~\ref{lem:t-ab}, we have
 \begin{align}
  \label{eq:9}
  \ds{H(x,t),H(y,t')}=&\ds{\eta_x(s),\eta_y(s')}\\
  \notag
  =&\ds{\eta_x(q),\eta_y(q)} + \lambda\abs{s-s'}+k_1\\
  \notag
  \leq &
  E'\left(\ds{\eta_x(\rho(T_x)),\eta_y(\rho(T_y))}
  +\lambda'\tilde{\theta}(\ds{\eta_x(\rho(T_x)),\eta_y(\rho(T_y))}\right) \\
  \notag
& + D' + \lambda(\tilde{\theta}(\ds{x,y})  + 2\abs{t-t'} + 2\delta + \tau) +k'_1.
 \end{align}

 Since $F$ is $\rho$-$T$-$\xcalL{X}_a^\infty$-$\xcalL{Y}_b^\infty$-equivariant,
 there exists $\xi_x,\xi_y\in \xcalL{Y}_b^\infty$ such that 
 \begin{align}
  \label{eq:xi-Fgm}
  \sup_{u\in\Rp }  \max\left\{  \ds{\xi_x(\rho(u)), F(\gamma_x(u))},
  \ds{\xi_y(\rho(u)),F(\gamma_y(u))}\right\}<T.
 \end{align}
 By Lemma~\ref{lem:LLoo-f}, $[\eta_x] = \partial F ([\gamma_x]) = [\xi_x]$ and 
 $[\eta_y] = \partial F ([\gamma_y]) = [\xi_y]$. So by Lemma~\ref{lem:HdistD},
 \begin{align}
  \label{eq:eta-xi}
  \sup_{u\in\Rp }  \max\left\{  \ds{\eta_x(\rho(u)),\xi_x(\rho(u)))},
  \ds{\eta_y(\rho(u)),\xi_y(\rho(u))}\right\}\leq D.
 \end{align}  
 Then we have
 \begin{align}
  \label{eq:11}
  \ds{\eta_x(\rho(T_x)),\eta_y(\rho(T_y))} \leq & 
  \ds{\eta_x(\rho(T_x)),F(\gamma_x(\rho(T_x)))} +
  \ds{F(\gamma_x(\rho(T_x))),F(\gamma_y(\rho(T_y)))} \\
  \notag &+
  \ds{F(\gamma_y(\rho(T_y))), \eta_y(\rho(T_y))} \\
  \notag
  \leq&   \ds{F(\gamma_x(\rho(T_x))),F(\gamma_y(\rho(T_y)))} + 2(T+D).
 \end{align}
 Set $T:=\min\{T_x,T_y\}$. Then by Lemma~\ref{lem:t-ab}
 \begin{align}
\label{eq:16}
  \ds{\gamma_x(\rho(T_x)),\gamma_y(\rho(T_y))} 
\leq &
  \ds{\gamma_x(\rho(T)),\gamma_y(\rho(T))} 
  + \lambda\abs{\rho(T_x)-\rho(T_y)} + k_1\\
\notag
\leq &  E(\ds{x,y} + \lambda\tilde{\theta}(\ds{x,y})+k_1)+D
  + \lambda(\tilde{\theta}(\ds{x,y})+\tau) + k_1.
 \end{align}

 The estimates ~(\ref{eq:9}),~(\ref{eq:11}),~(\ref{eq:16})
 and the assumption that $F$ is a large scale Lipschitz map imply
 $H$ is bornologous.
\end{proof}

\begin{lemma}
\label{lem:close}
 The map $\radp \partial F$ is coarsely homotopic to 
 the restriction $F|_{\Xvis}$.
\end{lemma}

\begin{proof}
 By Lemma~\ref{lem:H-coarse-map}, the map $H$ is a coarse map.
 For $(x,T_x)\in \Zvis$, we have 
 $H(x,T_x) = \eta_x(r(T_x^\epsilon)^{1/\epsilon}) = \radp \partial F$.
 It is enough to show that $H(-,0)$ is close to $F|_{\Xvis}$, that is, 
 $\sup_{x\in \Xvis}\ds{H(x,0),F(x)} <\infty$.

 Let $x\in \Xvis$. Then $H(x,0)=\eta_x(\rho(T_x))$. 
 Then by~(\ref{eq:xi-Fgm}) and~(\ref{eq:eta-xi}) in the proof of 
 Lemma~\ref{lem:H-coarse-map}, we have 
 $\ds{\eta_x(\rho(T_x)),F(\gamma_x(T_x))}\leq T+D$. Therefore 
 $H(-,0)$ is close to $F|_{\Xvis}$.
\end{proof}

\begin{proposition}
\label{prop:radf-chmtp}
 The map $\rad \partial F$ is coarsely homotopic to $F$.
\end{proposition}
\begin{proof}
 By Lemma~\ref{lem:close}, 
 $\rad \partial F= \radp \partial F \circ \varphi$ is coarsely homotopic to
 the composite $F\circ \varphi$. In \cite[Section 5.5]{FO-CCH}, it is proved 
 that $\varphi$ is coarsely homotopic to the identity. Therefore
 $\rad \partial F$ is coarsely homotopic to $F$.
\end{proof}

\bibliographystyle{amsplain} \bibliography{/Users/tomo/Library/tex/math}

\bigskip
\address{ Yuuhei Ezawa \endgraf
Department of Mathematical Science,
Tokyo Metropolitan University,
Minami-osawa Hachioji, Tokyo, 192-0397, Japan
}

\textit{E-mail address}: \texttt{ezawa.zawazawa@gmail.com}

\bigskip

\address{ Tomohiro Fukaya \endgraf
Department of Mathematical Science,
Tokyo Metropolitan University,
Minami-osawa Hachioji, Tokyo, 192-0397, Japan
}

\textit{E-mail address}: \texttt{tmhr@tmu.ac.jp}



\end{document}